\documentclass[11pt,a4paper,usenames,dvipsnames]{article}

\usepackage{url,amsmath,amssymb,latexsym,mathrsfs,comment,amsthm,enumerate,geometry,calc,ifthen,mathdots,textcomp,rotating}

\usepackage[T1]{fontenc}
\usepackage{wasysym}
\usepackage{stmaryrd}

\usepackage{cases}

\usepackage[noadjust]{cite}

\usepackage[colorlinks]{hyperref}

\usepackage[margin=10pt,font=small,labelfont=bf, labelsep=period]{caption}

\usepackage{makecell}

\usepackage[x11names, svgnames, rgb,table]{xcolor}
\usepackage{hhline}
\usepackage{multirow}

\usepackage{enumitem}

\usepackage[utf8]{inputenc}
\usepackage[T1]{fontenc}

\usepackage{tikz}
\usetikzlibrary{decorations.markings}
\usetikzlibrary{matrix}
\usepgflibrary{snakes,arrows,shapes}
\pgfdeclarelayer{background layer}
\pgfsetlayers{background layer,main}

\newcommand\nc\newcommand
\nc\rnc\renewcommand

\nc\Sub{\operatorname{\sf Sub}}
\nc\NSub{\operatorname{\sf NSub}}
\nc\Ht{\operatorname{\sf Ht}}
\nc\bH{{\bf H}}
\nc\es\varnothing
\nc\la\langle
\nc\ra\rangle
\nc\G{\mathscr G}
\nc\rG{\mathrel{\G}}
\nc\leqG{\operatorname{\leq_\G}}
\nc{\Mod}[1]{\ (\mathrm{mod}\ #1)}
\nc\da\downarrow
\nc\ua\uparrow

\nc\trans[1]{\left(\begin{smallmatrix}#1\end{smallmatrix}\right)}

\nc\ol\overline

\newcommand{\Eq}{\operatorname{\sf Eq}}
\newcommand{\Cong}{\operatorname{\sf Cong}}
\newcommand{\RCong}{\operatorname{\sf RCong}}
\newcommand{\LCong}{\operatorname{\sf LCong}}

\nc\JEnew[1]{{\color{black} #1}}
\nc\JErev[1]{{\color{black} #1}}
\newcommand{\NRrev}[1]{{\color{black} #1}}

\nc\al\alpha
\nc\be\beta
\nc\ga\gamma
\nc\de\delta
\rnc\th\theta
\nc\lam\lambda
\nc\om\omega
\nc\ze\zeta
\nc\si\sigma
\nc\Si\Sigma
\nc\Ga\Gamma
\nc\De\Delta
\nc\Th\Theta
\nc\Om\Omega
\nc\nab\nabla

\nc\AND{\qquad\text{and}\qquad}
\nc\ANDSIM{\qquad\text{and similarly}\qquad}
\nc\ANd{\quad\text{and}\quad}
\nc\COMMA{,\qquad}
\nc\COMMa{,\quad}
\nc\WHERE{\qquad\text{where}\qquad}

\rnc\iff{\ \Leftrightarrow\ }
\nc\IFf{\quad \Leftrightarrow\quad }
\nc\Iff{\ \ \Leftrightarrow\ \ }
\nc\IFF{\qquad \Leftrightarrow\qquad }
\rnc\implies{\ \Rightarrow\ }
\nc\IMPLIES{\qquad \Rightarrow\qquad }

\nc\set[2]{\{#1:#2\}}

\DeclareMathOperator{\ar}{ar}

\nc\bit{\begin{itemize}[label=\textbullet, leftmargin=5mm]}
\nc\eit{\end{itemize}}
\nc\ben{\begin{thmenumerate}}
\nc\bena{\begin{enumerate}[label=\textup{(\alph*)},leftmargin=10mm]}
\nc\een{\end{thmenumerate}}
\nc\eena{\end{enumerate}}

\nc\pf{\begin{proof}}
\nc\epf{\end{proof}}
\nc\epfres{\hfill\qed}
\nc\epfreseq{\tag*{\qed}}
\let\oldproofname=\proofname
\renewcommand{\proofname}{\rm\bf{\oldproofname}}
\nc{\pfitem}[1]{\medskip \noindent #1.}
\nc{\firstpfitem}[1]{#1.}

\renewcommand{\H}{\mathscr H}
\renewcommand{\L}{\mathscr L}
\newcommand{\R}{\mathscr R}
\newcommand{\D}{\mathscr D}
\newcommand{\J}{\mathscr J}
\newcommand{\K}{\mathscr K}
\nc\rH{\mathrel{\H}}
\nc\rL{\mathrel{\L}}
\nc\rR{\mathrel{\R}}
\nc\rD{\mathrel{\D}}
\nc\rJ{\mathrel{\J}}
\nc\rK{\mathrel{\K}}
\nc\rsi{\mathrel{\si}}

\nc\leqL{\leq_\L}
\nc\leqR{\leq_\R}
\nc\leqJ{\leq_\J}
\nc\leqH{\leq_\H}

\renewcommand{\P}{\mathcal P} 
\rnc\S{\mathcal S}
\nc\A{\mathcal A}
\nc\B{\mathcal B}
\nc\M{\mathcal M}
\nc\TL{\mathcal T\!\mathcal L}
\nc\F{\mathcal F}
\nc\I{\mathcal I}
\rnc\O{\mathcal O}
\newcommand{\T}{\mathcal T}
\newcommand{\PT}{\mathcal P\mathcal T} 
\newcommand{\PB}{\mathcal P\mathcal B} 
\renewcommand{\I}{\mathcal I}  

\newcommand{\GL}{\textup{GL}}

\newcommand{\FF}{\mathbb{F}}

\newcommand{\rank}{\operatorname{rank}}
\newcommand{\im}{\operatorname{im}}

\newcommand{\restr}{\mathord{\restriction}}
\nc\sub\subseteq
\nc\mt\mapsto
\nc\wh\widehat
\nc\normal\unlhd
\nc\sm\setminus
\newcommand{\smt}{\!\setminus\!}
\nc\mr\mathrel
\nc\pc[2]{(#1,#2)^\sharp}

\numberwithin{equation}{section}

\newtheorem{thm}[equation]{Theorem}
\newtheorem{lemma}[equation]{Lemma}
\newtheorem{cor}[equation]{Corollary}
\newtheorem{prop}[equation]{Proposition}

\theoremstyle{definition}
\newtheorem{rem}[equation]{Remark}
\newtheorem{defn}[equation]{Definition}

\newenvironment{thmenumerate}{
   \begin{enumerate}[label=\textup{(\roman*)}, widest=(5), leftmargin=10mm]}{
    \end{enumerate}}

\begin{document}

\title{Heights of one- and two-sided congruence lattices of semigroups\footnote{Supported by the Engineering and Physical Sciences Research Council [EP/S020616/1, EP/V002953/1 and EP/V003224/1] and the Australian Research Council [FT190100632].}~\vspace{-4ex}}

\date{}
\author{}

\maketitle
\begin{center}
{\large 
Matthew Brookes,%
\hspace{-.25em}\footnote{\label{fn:SC}Mathematical Institute, School of Mathematics and Statistics, University of St Andrews, St Andrews, Fife KY16 9SS, UK. {\it Emails:} {\tt mdgkb1@st-andrews.ac.uk}, {\tt jdm3@st-andrews.ac.uk}, {\tt nik.ruskuc@st-andrews.ac.uk}}
James East,%
\hspace{-.25em}\footnote{Centre for Research in Mathematics and Data Science, Western Sydney University, Locked Bag 1797, Penrith NSW 2751, Australia. {\it Email:} {\tt j.east@westernsydney.edu.au}}
Craig Miller,%
\hspace{-.25em}\footnote{Department of Mathematics, University of York, York YO10 5DD, UK. {\it Email:} {\tt craig.miller@york.ac.uk}}
James D.~Mitchell,\hspace{-.25em}\textsuperscript{\ref{fn:SC}}
Nik Ru\v{s}kuc\textsuperscript{\ref{fn:SC}}
}
\end{center}

\begin{abstract}
\noindent
The \emph{height} of a poset $P$ is the supremum of the cardinalities of chains in $P$.
The exact formula for the height of the subgroup lattice of the symmetric group $\S_n$ is known, as is an accurate asymptotic formula for the height of the subsemigroup lattice of the full transformation monoid $\T_n$.
Motivated by the related question of determining the heights of the lattices of left and right congruences of $\T_n$, 
\NRrev{and deploying the framework of unary algebras and semigroup actions,}
we develop a general method for computing the heights of lattices of both one- and two-sided congruences for semigroups.
We apply this theory to obtain exact height formulae for several monoids of transformations, matrices and partitions, including: 
the full transformation monoid $\T_n$, the partial transformation monoid~$\PT_n$, the symmetric inverse monoid $\I_n$, the monoid of order-preserving transformations $\O_n$, the full matrix monoid $\M(n,q)$, the partition monoid $\P_n$, the Brauer monoid $\B_n$ and the Temperley--Lieb monoid $\TL_n$.
\medskip

\noindent
\emph{Keywords}: Semigroup, semigroup action, unary algebra, (left/right) congruence, congruence lattice, height, modular element,
 transformation and diagram monoids, Sch\"{u}tzenberger group.
\medskip

\noindent
MSC (2020): 20M10, 20M30, 20M20, 08A30, 08A60, 06B05.

\end{abstract}

\setcounter{tocdepth}{1}
\tableofcontents

\section{Introduction}

The \emph{height} of a poset $P$, denoted $\Ht(P)$, is the supremum of the cardinalities of chains in $P$.
In~\cite{CST1989} the height of the subgroup lattice $\Sub(\S_n)$ of the symmetric group $\S_n$ is computed:

\begin{thm}[Cameron, Solomon and Turull {\cite[Theorem 1]{CST1989}}]
\label{thm:HtSn}
For $n\geq1$, we have 
\[
\Ht(\Sub(\S_n)) = \lceil3n/2\rceil - b(n),
\]
where $b(n)$ is the number of $1$s in the binary expansion of $n$.
\end{thm}

The formula given in \cite{CST1989} in fact has an additional term $-1$. This discrepancy is caused by the (slightly) different definition of height (which they call length): while in \cite{CST1989} it is taken to mean the number of covering relationships in a longest chain (at least in the finite case), for us it is the number of elements.

The following \JEnew{result concerning} the subsemigroup lattice of the full transformation monoid~$\T_n$ was obtained in \cite{CGMP2017}:

\begin{thm}[Cameron, Gadoleau, Mitchell, Peresse {\cite[Theorem 4.1]{CGMP2017}}]\label{thm:HtSubTn}
For $n\geq1$, we have
\[
\Ht(\Sub(\T_n)) \geq e^{-2} n^n -2e^{-2}(1-e^{-1})n^{n-1/3}-o(n^{n-1/3}).
\]
\end{thm}

Theorem \ref{thm:HtSn} can be viewed from a different angle, by recalling that the subgroups of a group $G$ are in one-one correspondence with transitive permutation representations of $G$, or, equivalently, with left (or right) congruences of $G$. The notion of left congruences of a semigroup~$S$ will be introduced explicitly in Section \ref{sect:prelim}, but for now let us just recall that they
form a lattice, which we denote by $\LCong(S)$.

So, motivated by the two results above one may ask: what is $\Ht(\LCong(\T_n))$? 
Towards the end of this paper we will prove the following:

\begin{thm}
\label{th:HtLCTn}
For $n\geq1$, we have
\[
\Ht(\LCong(\T_n)) = \sum_{r=1}^n \binom nr \Ht(\Sub(\S_r)).
\]
\end{thm}

\NRrev{Whereas the left and right congruence lattices of a group 
are always isomorphic, 
the same is in general not true for semigroups.
\JErev{For example, $\LCong(\T_3)$ and $\RCong(\T_3)$ have size $120$ and~$287$, respectively.}
Thus, one can equally ask about the height of the lattice $\RCong(\T_n)$, where we obtain the following:}

\begin{thm}
\label{th:HtRCTn}
For $n\geq 2$, we have
\[
\Ht(\RCong(\T_n)) = 1+\sum_{r=1}^n S(n,r) \Ht(\Sub(\S_r)),
\]
where $S(n,r)$ is the Stirling number of the second kind.
\end{thm}

Comparing Theorems \ref{thm:HtSn}--\ref{th:HtRCTn} we observe that: 
\bit
\item
the height of the subgroup lattice (and hence of the left and right congruence lattices) of the symmetric group $\S_n$ grows linearly with $n$, i.e.~its degree;
\item
 the height of the subsemigroup lattice of the full transformation monoid $\T_n$ grows linearly with $n^n$, i.e.~its order; 
 \item
 the height of $\LCong(\T_n)$ grows linearly with $n2^n$;  
 this follows from the combinatorial identity
 $ \sum_{r=1}^n r\binom{n}{r}=n2^{n-1}$;
 \item 
 the height of $\RCong(\T_n)$ grows linearly with
 $\sum_{r=1}^n rS(n,r)={B(n+1)-B(n)}\sim B(n+1)$, where $B(n)$ is the $n$th Bell number; see OEIS,  \cite[A005493]{OEIS}.
\eit

We remark in passing that, for a finite or countably infinite semigroup $S$, the heights of $\Cong(S)$, $\LCong(S)$ and 
$\RCong(S)$ can  be as large as $|S|$ or $2^{\aleph_0}$, respectively.
This can be seen for instance by considering null semigroups, for which all equivalence relations are congruences.

The current paper is devoted to building a theoretical framework within which we will be able to prove results like Theorems \ref{th:HtLCTn} and \ref{th:HtRCTn}.
Here is a brief summary.
We begin our development in Sections \ref{sect:algebras} and \ref{sect:Rees} by introducing some elements of general (universal) algebra concerning quotients, principal factors and Rees subalgebras.
We also recall the notion of a modular element in a lattice, which is the key property that enables inductive proofs of our main formulae; see Proposition \ref{pro:srs}.
Then we specialise to unary algebras in Section \ref{sect:unary}, where we prove a formula expressing the height of the congruence lattice of a unary algebra $A$ in terms of heights of quotients corresponding to strongly connected components of $A$ (Theorem \ref{thm:G}).
All the subsequent results in the paper have this result as their foundation.
We then move on to consider our main topic in Sections \ref{sect:prelim}--\ref{sect:applications}, the heights of the one- and two-sided congruence lattices of a semigroup.
The fact that left congruences are intimately related with left $S$-acts (transformation representations), and that the latter are a special case of unary algebras, provides the necessary link with the preceding development,
and enables us to  
express the height of 
$\LCong(S)$ in terms of heights of left $S$-acts corresponding to the $\L$-classes of $S$ (Theorems \ref{thm:S} and \ref{thm:S2}). 
We then analyse such acts, and bring the associated Sch\"{u}tzenberger groups into play, in Propositions~\ref{pr:Hrest} and \ref{pr:sep}.
Putting these components together, we obtain a general formula for the height of $\LCong(S)$, under certain additional assumptions, in Theorem~\ref{thm:S3} and Corollary \ref{co:S3}, of which Theorem \ref{th:HtLCTn} is a special instance.
We then note that all the foregoing considerations hold for the lattice of right congruences of $S$, via an obvious left-right duality, and state the main results in Section \ref{sect:right}.
The lattice of two-sided congruences is treated in Section \ref{sect:two} by considering $S$ as an $(S,S)$-biact, yet another special case of unary algebras.
With the theoretical development now complete, the paper then culminates in Section \ref{sect:applications}, where we give explicit formulae for the heights of the lattices of left, right and two-sided congruences for many of the most important and well-studied families of monoids in the literature.
In addition to the full transformation monoid $\T_n$, these include the partial transformation monoid~$\PT_n$, the symmetric inverse monoid $\I_n$, the monoid of order-preserving transformations~$\O_n$, the full linear monoid~$\M(n,q)$, the partition monoid~$\P_n$, the Brauer monoid~$\B_n$ and the Temperley--Lieb monoid $\TL_n$.
The formulae are given in Table~\ref{tab:formulae}, and some calculated values in Table \ref{tab:small}.

%
%
%

\section{Algebras: congruences, $\G$-classes and principal factors}\label{sect:algebras}

In this section we introduce the basic concepts and notation for \JEnew{universal} algebras; for a more systematic introduction see \cite{BS1981}.

A \emph{signature} (or \emph{type}) of algebras is an indexed family $F=(f_i)_{i\in I}$ of symbols, and associated integers $\ar(f_i)\geq0$ called \emph{arities}.
An \emph{algebra} of type $F$ is a pair $\A=(A,F^\A)$, where $A$ is a set, and $F^\A=(f_i^\A)_{i\in I}$ is an indexed collection
of operations $f_i^\A: A^{r_i}\rightarrow A$ where $r_i=\ar(f_i)$; \JEnew{these are called the \emph{basic} operations of}
\NRrev{$\A$}.
When there is no danger of confusion we omit superscripts~$\A$, thus identifying operation symbols \JEnew{with the} operations they represent, and we also treat $F$ and~$F^\A$ as sets.
Furthermore, we often identify the algebra $\A$ with its underlying set $A$.
Note that in principle we allow an algebra to be empty, but this cannot be the case for an algebra whose signature contains constant (nullary) operations.

A \emph{subalgebra} of an algebra $\A$ is any algebra $\B=(B,F^\B)$ where $B\subseteq A$ and \JEnew{$f_i^\B=f_i^\A\restr_B$} for all $i\in I$.
We write $\B\leq\A$ to indicate that $\B$ is a subalgebra of $\A$.  The collection of all subalgebras of $\A$ forms a lattice under inclusion. This lattice is denoted by $\Sub(\A)$.

A \emph{congruence} of an algebra \NRrev{$(A,F)$} is an equivalence relation $\si$ on $A$ that is \emph{compatible} with each operation $f$ from $F$, in the sense that:
\[
(a_1,b_1),\ldots,(a_n,b_n)\in\si \implies (f(a_1,\ldots,a_n) , f(b_1,\ldots,b_n))\in\si.
\]
The set $\Cong(\NRrev{A})$ of all congruences of $\NRrev{A}$ is a lattice under inclusion, called the \emph{congruence lattice} of $\NRrev{A}$.  The bottom and top elements of $\Cong(\NRrev{A})$ are the \emph{trivial (a.k.a.~diagonal)} and \emph{universal} relations:
\[
\De_A = \set{(a,a)}{a\in A} \AND \nab_A = A\times A,
\]
respectively.  The \emph{quotient} $A/\si$ of $A$ by $\si$ is the set $\set{ [a]}{a\in A}$ of $\si$-classes with the induced operations
$f^{A/\si}([a_1],\dots,[a_n]):=[f^A(a_1,\dots,a_n)]$.


Consider an algebra $A$.  For $a\in A$, we write $\la a\ra$ for the subalgebra of $A$ generated by $a$;
it consists of all elements of the form $f(a,\dots,a)$, where $f$ is a \emph{term operation}, i.e.~an operation obtained by composing the basic operations.
In particular, $\la a\ra$ contains all constants of $A$, i.e. the images of all nullary operations.  We define a preorder ${\leqG}={\leq_{\G^A}}$ on~$A$, which then yields an equivalence relation $\G=\G^A$ on $A$ and a partial order 
$\leq$ on the
set $A/\G=\set{ G_a}{ a\in A}$ of $\G$-classes, as follows:
\NRrev{\begin{equation}\label{eq:G}
a\leqG b \iff \la a\ra \sub \la b\ra,\quad
a\rG b \iff \la a\ra = \la b\ra,\quad
G_a\leq G_b \iff a\leqG b\qquad
\JEnew{\text{for $a,b\in A$.}}
\end{equation}}
The well-known Green's relations $\R$, $\L$, $\J$ \JEnew{and the associated pre-orders} on a semigroup will later be seen as special instances of these relations, which explains the choice of notation here.

\begin{lemma}\label{lem:GAB}
If $A$ is an algebra, and if $B\leq A$ is a subalgebra, then
\[
{\leq_{\G^B}} = {\leq_{\G^A}}\restr_B \AND \G^B = \G^A\restr_B.  
\]
\end{lemma}

\begin{proof}
This follows from the  fact that the subalgebra of $B$ generated by $b\in B$ is the same as the subalgebra of $A$ generated by~$b$.
\end{proof}

Consider an algebra $A$ and an arbitrary subset $X\sub A$.  
Let $0$ be a symbol not belonging to~$A$, and let 
\[
\ol X = X\cup\{0\}.
\]
We can turn $\ol X$ into an algebra of the same signature as $A$ in the following way, by setting, for any $n$-ary basic operation symbol $f$, and for any $x_1,\ldots,x_n\in\ol X$:
\[
\ol f(x_1,\ldots,x_n) := \begin{cases}
f(x_1,\ldots,x_n) &\text{if $x_1,\ldots,x_n,f(x_1,\ldots,x_n)\in X$,}\\
0 &\text{otherwise.}
\end{cases}
\]
We call $\ol X$ the \emph{principal factor} \JEnew{associated to} $X$.
We note that $\ol X$ need not be an algebra of the same `kind' as $A$; e.g.~if $A$ is a group and $X\neq\es$, then $\ol X$ is never a group.

\begin{lemma}\label{lem:a0}
Suppose $\si\in\Cong(\ol G)$, where $G$ is a $\G$-class of an algebra $A$ \JEnew{with no} \NRrev{constants}.  If $a\mr\si0$ for some $a\in G$, then $\si=\nab_{\ol G}$.
\end{lemma}

\pf
It suffices to show that $b\mr\si0$ for all $b\in G$, so fix some such $b$.  Since $a\rG b$, we have $b\in\la a\ra$,
and so $b=f(a,\ldots,a)$ for some term function $f$ of $A$ \NRrev{of arity $>0$ because $A$ has no constants.
 Using} $a,b\in G$, we have
\[
b = f(a,\ldots,a) = \ol f(a,\ldots,a) \mr\si \ol f(0,\ldots,0) = 0.  \qedhere
\]
\epf

\JEnew{
Lemma \ref{lem:a0} need not be true if $A$ has constant/nullary operations.  For example, consider the case in which $A=\{a,b\}$, \NRrev{with~$a$ and~$b$ both constants, and no other operations.}  Since $A=\la a\ra = \la b\ra$, it follows that $A$ is itself a $\G$-class.  But the equivalence $\si$ on $\ol A = A\cup\{0\}$ with $\si$-classes $\{a,0\}$ and $\{b\}$ is a non-universal congruence, yet $a\mr\si0$.
}

In the following statement, for a poset $P$, we denote by $P^\top = P\cup\{\top\}$ the poset obtained by adjoining a new top element $\top$ to $P$.

\begin{lemma}\label{lem:GS}
If a $\G$-class $G$ of an algebra $A$ \NRrev{with no constants} is also a subalgebra of~$A$, then
\[
\Cong(\ol G) \cong \Cong(G)^\top \AND \Ht(\Cong(\ol G)) = \Ht(\Cong(G))+1.
\]
\end{lemma}

\pf
It suffices to prove the first claim. For $\sigma\in\Cong(G)$ define $\overline{\sigma}:=\sigma\cup\{(0,0)\}$.
We prove the result by showing that
\[
\Cong(\ol G) =\set{ \overline{\sigma}}{ \sigma\in\Cong(G)} \cup \{\nab_{\ol G}\} .  
\]

\pfitem{($\supseteq$)}
For $\si\in\Cong(G)$, clearly $\overline{\sigma}$ is an equivalence on $\ol G$.  To check compatibility, suppose $f:A^n\to A$ is an operation of $A$, and suppose $(a_1,b_1),\ldots,(a_n,b_n)\in\overline{\sigma}$.  
From the definition of~$\overline{\sigma}$ we have $a_i=0 \iff b_i=0$ ($i=1,\dots,n$).
If $a_i=0$ for some $i$ then 
\[
(\ol f(a_1,\ldots,a_n),\ol f(b_1,\ldots,b_n)) = (0,0) \in \ol\si.
\]
Otherwise, $a_i,b_i\not=0$ for all $i$, in which case
\[
(\ol f(a_1,\ldots,a_n),\ol f(b_1,\ldots,b_n))=(f(a_1,\ldots,a_n),f(b_1,\ldots,b_n))\in\si\sub\ol\si.
\]

\pfitem{($\sub$)}
By Lemma \ref{lem:a0}, for $\nabla_{\overline{G}}\neq \tau\in\Cong(\overline{G})$ we have $\tau=\tau\restr_G\cup\{(0,0)\}$.
But $G\leq\overline{G}$ (as $G\leq A$ by assumption), and so $\tau\restr_G\in\Cong(G)$, which proves the inclusion, and
hence the lemma.
\epf

\section{Rees subalgebras, ideals and quotients}\label{sect:Rees}

In this section we develop  a general framework within which we can express the height of the congruence lattice of an algebra as a sum involving a special kind of subalgebra and an associated quotient; see Proposition \ref{pro:addeq}.

We denote the \JEnew{down-set (a.k.a.~principal ideal) and up-set (a.k.a.~principal filter)} of an element~$x$ of a lattice $L$ by
\[
x^\da = \set{a\in L}{a\leq x} \AND x^\ua = \set{a\in L}{a\geq x} ,
\]
respectively.  For any chains $C_1\sub x^\downarrow$ and $C_2\sub x^\uparrow$, their union $C_1\cup C_2$ is a chain in~$L$, with $|C_1\cap C_2|\leq 1$.  It quickly follows that
\begin{equation}\label{eq:Lxx}
\Ht(L) \geq \Ht(x^\downarrow)+\Ht(x^\uparrow)-1 \qquad\text{for all $x\in L$.}
\end{equation}
The next result establishes a sufficient condition for this to become an equality.
Following Stanley~\cite{Stanley1971}, we say that an element $x$ in a lattice $L$ is \emph{modular}
if it forms a modular pair (see \cite[Section V.2.5]{Gratzer2011}) with every other element of $L$, i.e.~if
\[
a\vee (x\wedge b)=(a\vee x)\wedge b \qquad\text{for all $a,b\in L$ with $a\leq b$.}
\]
\JErev{In practice, to show that $x$ is modular, one only needs to check that $a\vee (x\wedge b) \geq (a\vee x)\wedge b$, as the reverse inequality always holds (when $a\leq b$).}

\begin{prop}
\label{pro:srs}
If $x$ is a modular element in a lattice $L$,
then every chain in $L$ embeds into the direct product $x^\downarrow\times x^\uparrow$, and consequently
$\Ht(L)=\Ht(x^\downarrow)+\Ht(x^\uparrow)-1$.
\end{prop}

\begin{proof}
Let $C$ be a chain in $L$, and define a mapping
\[
f:C\rightarrow x^\downarrow\times x^\uparrow;\ c\mapsto (c\wedge x,c\vee x).
\]
We need to show that $a\leq b\iff f(a)\leq f(b)$ for all $a,b\in C$.
The direct implication is clear.
For the converse,
suppose $a,b\in C$ satisfy $f(a)\leq f(b)$, i.e.~$a\wedge x\leq  b\wedge x$ and $a\vee x\leq b\vee x$.
Since $C$ is a chain, we have $a\leq b$ or $b\leq a$. In the former case we are done.
If $b\leq a$, then
using modularity of $x$, we have
\[
a=(a\vee x)\wedge a\leq (b\vee x)\wedge a=b\vee (x\wedge a)\leq b\vee (x\wedge b)=b,
\]
as required.

Having established the first claim, it follows that $\Ht(L)\leq \Ht(x^\downarrow\times x^\uparrow)=\Ht(x^\downarrow)+\Ht(x^\uparrow)-1$.
Combined with \eqref{eq:Lxx}, this completes the proof.
\end{proof}

Proposition \ref{pro:srs} provides a first tool for calculating the height of a congruence lattice $\Cong(A)$; one hopes to identify a congruence $\si$ of $A$ that is modular in the congruence lattice, and then $\Ht(\Cong(A))$ can be expressed in terms of the heights of $\si^\da=[\De_A,\si]$ and $\si^\ua=[\si,\nab_A]$.
In many classical varieties of algebras, such as groups and rings,
every congruence is modular in the congruence lattice\NRrev{, e.g.~see \cite[Exercises I.3.5 and I.3.6]{BS1981}}.
This is not true for algebras in general, or for unary algebras or semigroups in particular.
\JErev{We now identify a special kind of subalgebra associated with a special kind of congruence; the congruences arising turn out to always be modular, and
will prove very useful in what follows.}

Specifically, if $B\leq A$ is such that the relation $\rho_B:=\nabla_B\cup\Delta_A$ is a congruence on $A$, we say that \JErev{$B$ is a \emph{Rees subalgebra}, and $\rho_B$ a \emph{Rees congruence}; the \emph{Rees quotient} associated with $B$ is $A/B:=A/\rho_B$}.
These concepts were introduced in \cite{Tichy1981}. The definition there was couched in different (but equivalent) terms, and our defining property was established as \cite[Proposition~1.2]{Tichy1981}.  

There is a link between Rees quotients and principal factors from the previous section, in the special case of ideals.
Specifically, $B\leq A$ is an \emph{ideal} if for each $n$-ary operation $f$ and all $a_1,\dots,a_n\in A$, at least one of which belongs to $B$, we have
$f(a_1,\dots,a_n)\in B$.
We remark in passing that this definition generalises the notion of ideals in semigroups, but not in rings.
Immediate from the definition is the following: 

\begin{lemma}\label{lem:A/B}
Every ideal $B\leq A$ is a Rees subalgebra, and 
\[
A/B\cong \begin{cases} A & \text{if } B=\es\\ \overline{A\smt B} & \text{if } B\neq \es.\end{cases}
\epfreseq
\]
\end{lemma}

\JEnew{The converse of Lemma \ref{lem:A/B} is not true.  For example, any one-element subalgebra is a Rees subalgebra (whose associated Rees congruence is the trivial relation), but need not be an ideal.  Concretely, $\{1\}$ is a Rees subalgebra of a group, but is only an ideal if the group is trivial.}


\begin{prop}
\label{pr:reesinter}
If $A$ is an algebra, and if $B\leq A$ is a Rees subalgebra, then the Rees congruence~$\rho_B$ is modular in the lattice $\Cong(A)$, and hence
\[
\Ht(\Cong(A))=\Ht([\Delta_A,\rho_B])+\Ht(\Cong(A/B))-1.
\]
\end{prop}

\begin{proof}
The second assertion follows from the first and Proposition \ref{pro:srs}, upon noticing that $[\rho_B,\nabla_A]\cong \Cong(A/B)$
by the Correspondence Theorem (see \cite[Theorem 6.20]{BS1981} for instance).

To prove the first assertion, suppose $\sigma_1,\sigma_2\in\Cong(A)$ satisfy $\sigma_1\subseteq\sigma_2$.
\JEnew{As noted just before Proposition \ref{pro:srs}, we just need to show that $(\sigma_1\vee\rho_B)\cap\sigma_2 \sub \sigma_1\vee(\rho_B\cap \sigma_2)$.
%
To do so, suppose $(a,b)\in(\sigma_1\vee\rho_B)\cap\sigma_2$.}
Then $(a,b)\in\sigma_2$, and there exist $k\geq 1$ and $c_1,d_1,\dots,c_k,d_k\in A$ such that
\[
a  \mr\si_1 c_1 \mr\rho_B d_1 \mr\si_1 c_2 \mr\rho_B d_2 \mr\si_1\dots\mr\si_1 c_k \mr\rho_B d_k\mr\si_1 b.
\]
Assume that $k(\geq1)$ is as small as possible with respect to the existence of such a sequence.
We claim that $k=1$. To prove this, suppose instead that $k\geq 2$. Minimality of $k$ implies that $c_i\neq d_i$ for all $i=1,\dots,k$.
Hence, by the definition of $\rho_B$, we must have $c_i,d_i\in B$ for all $i$.
But then $a \mr\si_1  c_1 \mr\rho_B d_k\mr\si_1 b$, which is a sequence of the same form and of length $1$.
This contradiction confirms that $k=1$. So now
\begin{equation}
\label{eq:acdb}
a\mr\si_1 c_1\mr\rho_B d_1\mr\si_1  b.
\end{equation}
But since $\sigma_1\subseteq \sigma_2$, it follows that
$c_1 \mr\si_2 a \mr\si_2 b\mr\si_2 d_1$, and hence $(c_1,d_1)\in \rho_B\cap\sigma_2$.
The sequence~\eqref{eq:acdb} now implies $(a,b)\in \sigma_1\vee(\rho_B\cap\sigma_2)$, completing the proof.
\end{proof}

In what follows we establish  a condition under which the term
$\Ht([\Delta_A,\rho_B])$ in Proposition~\ref{pr:reesinter} can be replaced by $\Ht(\Cong(B))$.
The next result shows that $\Ht([\Delta_A,\rho_B])\leq\Ht(\Cong(B))$ always holds.

\begin{lemma}
\label{la:htBr}
For  a Rees subalgebra $B\leq A$,
the mapping
\[
[\Delta_A,\rho_B]\rightarrow \Cong(B);\ \sigma\mapsto\sigma\restr_B
\]
 is an embedding of lattices.
\end{lemma}

\begin{proof}
For any $\sigma\in [\Delta_A,\rho_B]$ we have $\sigma=\sigma\restr_B\cup \Delta_{A\setminus B}$.
Hence, for $\sigma_1,\sigma_2\in [\Delta_A,\rho_B]$, we have 
$\sigma_1\subseteq\sigma_2\iff \sigma_1\restr_B\subseteq\sigma_2\restr_B$,
and the lemma follows.
\end{proof}

As usual, we say a subalgebra $B\leq A$ has the \emph{congruence extension property}, \JEnew{or that $B$ is a \emph{CEP subalgebra}}, if 
for every $\sigma\in\Cong(B)$ there exists $\sigma'\in\Cong(A)$ extending it, i.e.~$\sigma=\sigma'\restr_{B}$.
We say that~$B$ has the \emph{diagonal congruence extension property}, \JEnew{or that $B$ is a \emph{$\De$-CEP subalgebra}}, if  
for every $\sigma\in\Cong(B)$ we have $\sigma\cup\Delta_A\in\Cong(A)$.
It is easy to show that $B$ is a $\De$-CEP subalgebra if and only if it is \JEnew{both a CEP and Rees subalgebra}.

\begin{prop}
\label{pro:addeq}
Let $A$ be an algebra.
\begin{thmenumerate}
\item
\label{it:addeq1}
If $B $ is  a $\Delta$-CEP subalgebra of $A$, then $\Cong(B)\cong [\Delta_A,\rho_B]$, and hence
\[
\Ht(\Cong(A))=\Ht(\Cong(B))+\Ht(\Cong(A/B))-1.
\]
\item
\label{it:addeq2}
If $B_0\leq B_1\leq\dots\leq B_k=A$ is \NRrev{an ascending} chain of $\Delta$-CEP subalgebras of $A$, then
\[
\Ht(\Cong(A))=\Ht(\Cong(B_0))+\sum_{r=1}^{k} \Ht(\Cong(B_{r}/ B_{r-1}))-k.
\]
\end{thmenumerate}
\end{prop}

\begin{proof}
\firstpfitem{\ref{it:addeq1}}
Given Proposition \ref{pr:reesinter} and Lemma \ref{la:htBr}, it suffices to show
that the mapping given in the latter is surjective.
This in turn follows from the fact that $(\sigma\cup\Delta_A)\restr_B=\sigma$ for all $\sigma\in\Cong(B)$.

\pfitem{\ref{it:addeq2}}
This follows by induction on $k$.
The case $k=0$ is vacuous.
The inductive step relies on part~\ref{it:addeq1} and the observation that $B_0\leq\dots\leq B_{k-1}$ is a chain of $\Delta$-CEP subalgebras of $B_{k-1}$.
\end{proof}

Analogous statements to Proposition \ref{pro:addeq}\ref{it:addeq1} were given in \cite[Lemma 2.1]{CST1989} for subgroup lattices of groups, 
and in \cite[Proposition 3.1]{CGMP2017} for subsemigroup lattices of semigroups:
\bit
\item $\Ht(\Sub(G)) = \Ht(\Sub(N)) + \Ht(\Sub(G/N)) - 1$ for any normal subgroup $N$ of a group $G$.
\item $\Ht(\Sub(S)) = \Ht(\Sub(I)) + \Ht(\Sub(S/I)) - 1$ for any ideal $I$ of a semigroup $S$.
\eit
We mention in passing that using our terminology, the above statements can be proved by showing that $N$ is modular in $\Sub(G)$, and $I$ is modular in $\Sub(S)$.
The former is well known and easy to prove. The latter is equally easy to see.
Indeed, if $I$ is an ideal of a semigroup $S$ and if $U\leq V\leq S$ are subsemigroups, then $U\vee I= U\cup I$ \JEnew{(the join is in $\Sub(S)$)}, so that
\[
(U\vee I)\cap V=(U\cup I)\cap V=(U\cap V)\cup (I\cap V)=U\cup(I\cap V)=U\vee (I\cap V),
\]
where the last equality holds because $U\cup(I\cap V)\leq S$.

\section{Unary algebras: a height formula}\label{sect:unary}

We now put the developments from the previous two sections into the context of unary algebras, i.e.~algebras in which all basic operations have arity $1$.  Our main objective here is to obtain a formula that expresses the height of such an algebra in terms of the heights of the principal factors corresponding to its $\G$-classes, in the case that there are finitely many such classes.  This will be achieved in Theorem \ref{thm:G}, which will form the basis for several subsequent results on semigroups in Sections \ref{sect:left1}--\ref{sect:two}.
 
 We begin by showing that Proposition \ref{pro:addeq} always applies in unary algebras.

\begin{lemma}\label{lem:rhoB}
Any subalgebra of a unary algebra $A$ is \JErev{both an ideal and a $\Delta$-CEP subalgebra (and hence a Rees subalgebra)}.
\end{lemma}

\begin{proof}
For the first assertion, observe that the definition of a subalgebra and an ideal coincide in the unary case.

For the second assertion, let $B\leq A$ and $\sigma\in\Cong(B)$.
It is clear that $\sigma^A:=\sigma\cup\Delta_A$ is an equivalence relation.
To show compatibility, let $(a,b)\in \sigma^A$, and let $f$ be a basic (unary) operation.
If $(a,b)\in \sigma$ then $(f(a),f(b))\in\sigma\subseteq\sigma^A$, because $\sigma$ is a congruence on~$B$.
Otherwise, $a=b$, hence $f(a)=f(b)$, and so $(f(a),f(b))\in\Delta_A\subseteq \sigma^A$.
%
\end{proof}


We now bring in the $\G$-relations of a unary algebra $A$.
In this case, for $a,b\in A$ we have $a\leq_\G b$ if and only if $a=f(b)$ for some \JEnew{term operation $f$} of $A$.
The $\G$-classes of~$A$ are then the strongly connected components of the directed graph associated with $A$, where $a\rightarrow b$ if and only if $f(a)=b$ for some basic operation $f$.
The proof of our main result for this section will be by induction on the number of $\G$-classes. To make it work, we need some more information concerning maximal $\G$-classes in the $\leq$ order from \eqref{eq:G}.

\begin{lemma}\label{lem:G}
Let $G$ be a maximal $\G$-class of a unary algebra $A$, and let $B=A\smt G$.  
\ben
\item \label{lG1} $B$ is a subalgebra of $A$.
\item \label{lG2} If $B\not=\es$, then $A/B \cong \ol G$.
\een
\end{lemma}

\pf
\firstpfitem{\ref{lG1}}  Let $a\in B$, and let $f$ be an operation of $A$.  We must show that $f(a)\in B$, i.e.~that $f(a)\not\in G$.  To do so, suppose to the contrary that $f(a)\in G$.  Then $G = G_{f(a)} \leq G_a$ (as $f(a)\in\la a\ra$), so it follows from maximality of $G$ that $G=G_a$.  But this implies $a\in G$, contradicting $a\in B=A\smt G$.

\pfitem{\ref{lG2}}  \JEnew{By part \ref{lG1} and Lemma \ref{lem:rhoB}, $B$ is an ideal of $A$.  The claim now follows from Lemma \ref{lem:A/B}.}
\epf

\begin{cor}\label{cor:G}
If $G$ is a maximal $\G$-class of a unary algebra $A$, and $B=A\smt G$, then
\[
\Ht(\Cong(A/B))=\begin{cases} \Ht(\Cong(\overline{G})) & \text{if } G\neq A\\  \Ht(\Cong(\overline{G}))-1 & \text{if } G= A.\end{cases}
\]
\end{cor}

\pf
This follows immediately from Lemmas \ref{lem:G}\ref{lG2} and \ref{lem:GS} (and $A/\es\cong A$).
\epf

We are now ready for the main result of this section.

\begin{thm}\label{thm:G}
If $A$ is a unary algebra with finitely many $\G$-classes, $G_1,\ldots,G_k$, then
\[
\Ht(\Cong(A)) = \sum_{r=1}^k \Ht(\Cong(\ol G_r)) - k.
\]
\end{thm}

\pf
Without loss of generality, assume that for every $r\in\{1,\dots,k\}$, the class $G_r$ is minimal in $\{G_r,\dots, G_k\}$.
This means that for  every $x\in G_r$, and every operation $f$ of $A$, we have ${f(x)\in G_1\cup\dots\cup G_r}$.
It follows (keeping Lemma \ref{lem:GAB} in mind) that for every $r=1,\dots,k$ the
set $B_r:=G_1\cup\dots\cup G_r$ is a subalgebra of $A$ with a maximal $\G$-class $G_r$.
We therefore have a chain of ($\De$-CEP) subalgebras ${\es=B_0<B_1<\dots<B_k=A}$ (cf.~Lemma \ref{lem:rhoB}).  Combining 
Proposition~\ref{pro:addeq}\ref{it:addeq2} and Corollary \ref{cor:G} we obtain
\begin{align*}
\Ht(\Cong(A)) &=
\Ht(\Cong(B_0))+\sum_{r=1}^{k} \Ht(\Cong(B_{r}/ B_{r-1}))-k\\
& =
1 + \Ht(\Cong(\overline{G}_1))-1+\sum_{r=2}^k \Ht(\Cong(\overline{G}_r)) -k
=\sum_{r=1}^k \Ht(\Cong(\overline{G}_r))-k.  
\end{align*}
\JEnew{(Note here that the unique congruence on the empty algebra $B_0=\es$ is the empty relation, so that $\Cong(B_0) = \{\es\}$ has height $1$.)}
\epf

\section{Preliminaries on semigroups}\label{sect:prelim}

In this section we rapidly review the concepts and basic facts from the theory of semigroups that will be used in the remainder of the paper. For more background see \cite{Howie1995,CPbook}.

\subsection{Basics: left, right and two-sided congruences}

A \emph{semigroup}  is an algebra with a single associative binary operation, usually denoted by juxtaposition. 
 A \emph{monoid} is a semigroup with identity.  As usual, we write $S^1$ for the \emph{monoid completion} of a semigroup $S$; so $S^1=S$ if $S$ is a monoid, and otherwise $S^1=S\cup\{1\}$ where $1$ is a symbol not belonging to $S$, acting as an adjoined identity.

An equivalence relation $\sigma$ on a semigroup $S$ is said to be a \emph{left congruence} if it is compatible with left multiplication:
\[
(x,y)\in\si \implies (ax,ay)\in\si \qquad\text{for all $a,x,y\in S$.}
\]
\emph{Right congruences} are defined dually.
It is an easy fact that $\sigma$ is a congruence (as defined in Section \ref{sect:algebras}) if and only if it is both a left and a right congruence.
The sets 
\[
\LCong(S),\quad \RCong(S)\quad \text{and}\quad \Cong(S)
\]
of all left, right and two-sided congruences are all lattices under inclusion.  Each is a sublattice of the lattice \JEnew{$\Eq(S)$ of all equivalence relations on $S$}.

The left congruences of a \emph{group} $G$ are precisely the equivalence relations
whose equivalence classes are the left cosets of a subgroup, i.e.~relations of the form
\[
\lam_H = \set{(x,y)\in G\times G}{x^{-1}y\in H} \qquad \text{for }H\leq G.
\]
Analogous statements hold for the right congruences, using right cosets.
The two-sided congruences of $G$ are precisely the relations $\lam_N$ for any \emph{normal} subgroup $N\normal G$.  The next result quickly follows, with $\Sub(G)$ and $\NSub(G)$ denoting the lattices of subgroups and normal subgroups of~$G$, respectively.

\begin{lemma}\label{lem:CongG}
If $G$ is a group, then
\[
\LCong(G)\cong\RCong(G)\cong\Sub(G) \AND \Cong(G)\cong\NSub(G).  \epfreseq
\]
\end{lemma}

\subsection{Green's relations}

\emph{Green's preorders and equivalences} \cite{Green1951} on a semigroup $S$ are defined as follows.  For $x,y\in S$, we have
\begin{align*}
x \leqL y &\iff S^1x\sub S^1y  &  \quad x \leqR y &\iff xS^1\sub yS^1  & \quad x \leqJ y &\iff S^1xS^1 \sub S^1yS^1 \\
&\iff x\in S^1y, & &\iff x\in yS^1 , &  &\iff x\in S^1yS^1 .
\end{align*}
These preorders induce the equivalences
\[
\L = {\leqL}\cap{\geq_\L} \COMMA \R = {\leqR}\cap{\geq_\R} \AND \J = {\leqJ}\cap{\geq_\J}.
\]
The remaining two Green's equivalences are
\[
\H = \L\cap\R \AND \D = \L\vee\R,
\]
where the latter denotes the join in the lattice $\Eq(S)$.  It is well known that $\D=\L\circ\R=\R\circ\L$
\NRrev{(where $\circ$ denotes the usual composition of binary relations)}, and that $\D=\J$ if $S$ is finite.  
We denote the $\L$-class of an element $x\in S$ by $L_x$, and similarly for $\R$-, $\J$-, $\H$- and $\D$-classes.  If $\K$ is any of $\L$, $\R$ or $\J$, then the set $S/\K = \set{K_x}{x\in S}$ of all $\K$-classes of $S$ has an induced partial order $\leq$, given by
\[
K_x \leq K_y \iff x\leq_\K y \qquad\text{for all $x,y\in S$.}
\]

One of the most important basic results concerning Green's relations is the following; for a proof, see for example \cite[Lemma 2.2.1]{Howie1995}.  

\begin{lemma}[Green's Lemma]\label{lem:GL}
Let $x$ and $y$ be $\R$-related elements of a semigroup $S$, so that $y=xs$ and $x=yt$ for some $s,t\in S^1$.  Then the maps
\[
L_x\to L_y;\ u\mt us \AND L_y\to L_x;\ v\mt vt
\]
are mutually inverse bijections.  Moreover, these maps restrict to mutually inverse bijections
\[
H_x\to H_y;\ u\mt us \AND H_y\to H_x;\ v\mt vt.  \epfreseq
\]
\end{lemma}

\JErev{There is a left-right dual of the previous result, whose statement we omit, but throughout the article we will frequently refer to both formulations as Green's Lemma.  We note that Green's Lemma, and many of the other coming results and arguments concerning Green's relations, can be conveniently visualised through the use of so-called \emph{egg-box diagrams}, as explained for example in \cite[Chapter~2]{Howie1995}.}

A $\J$-class $J$ of a semigroup $S$ is \emph{stable} if
\[
x \rJ sx \implies x \rL sx \AND x \rJ xs \implies x \rR xs \qquad\text{for all $x\in J$ and $s\in S^1$.}
\]
A stable $\J$-class is in fact a $\D$-class \cite[Proposition 2.3.9]{Lallement1979}.  A semigroup $S$ is \emph{stable} if every $\J$-class is stable.  All finite semigroups are stable \cite[Theorem A.2.4]{RS2009}.  There is an obvious left-right dual of the next result, but again we will not state it.

\begin{lemma}\label{lem:stab}
Let $S$ be a semigroup, let $x,y\in S$ and $s\in S^1$, and suppose $x\rR y$.  Then
\ben
\item \label{stab1} $sx\rL x \iff sy\rL y$,
\item \label{stab2} if the $\J$-class containing $x,y$ is stable, then $sx\rJ x \iff sx\rL x \iff sy\rL y \iff sy\rJ y$.
\een
\end{lemma}

\pf
\firstpfitem{\ref{stab1}}  By symmetry it suffices to establish the forwards implication.  \JEnew{Since $x\rR y$, we have $y=xt$ for some $t\in S^1$.  Since $\L$ is a right congruence, $sx\rL x$ implies $sxt \rL xt$, i.e.~$sy \rL y$.}


\pfitem{\ref{stab2}}  This follows immediately from \ref{stab1} and stability.
\epf

\subsection{Acts}

Let $S$ and $T$ be semigroups.
A \emph{left $S$-act} is a unary algebra $A$ with basic operations 
$\set{\lam_s}{s\in S}$ satisfying the law $\lam_s\lam_{s'}=\lam_{ss'}$ for all $s,s'\in S$.  We typically write $s\cdot a=\lam_s(a)$ for $s\in S$ and $a\in A$.  In this way, the defining law says that
$s\cdot (s'\cdot a)=(ss')\cdot a$ for all $s,s'\in S$ and $a\in A$.
Dually,
a \emph{right $T$-act} is a unary algebra $A$ with basic operations $\rho_t:A\to A;\ a\mapsto a\cdot t$ ($t\in T$), 
satisfying
$(a\cdot t)\cdot t' = a\cdot (tt')$ for all $t,t'\in T$ and all $a\in A$.
We say that $A$ is an \emph{$(S,T)$-biact} if it is simultaneously a left $S$-act and a right $T$-act, and each $\lambda_s$ commutes with each $\rho_t$, i.e.~if $(s\cdot a)\cdot t=s\cdot (a\cdot t)$ for all $s\in S$, $t\in T$, $a\in A$.

Since acts are (unary) algebras, we can talk about their subacts and congruences, as well as their $\G$-relations and associated principal factors.

Next we outline a link between actions and congruences that is  vital for the viewpoint of this paper.
Let $S$ be a semigroup. Then $S$ can be considered as a left $S$-act, right $S$-act and $(S,S)$-biact, where $\lambda_s$ and $\rho_s$ are the left and right multiplications by $s$.
We denote these three acts by ${}^S\!S$, $S^S$ and ${}^S\!S^S$, respectively, and record the following immediate observation.

\begin{lemma}
\label{la:SSS}
For any semigroup $S$:
\begin{thmenumerate}
\item
\label{it:SSS1}
$\LCong(S)=\Cong({}^S\!S)$, $\RCong(S)=\Cong(S^S)$ and $\Cong(S)=\Cong({}^S\!S^S)$.
\item
\label{it:SSS2}
The $\G$-equivalences of ${}^S\!S$, $S^S$ and ${}^S\!S^S$ are respectively  Green's relations $\L$, $\R$ and $\J$ of~$S$.  \epfres
\end{thmenumerate}
\end{lemma}

\begin{rem}
Unary algebras, including semigroup acts, can naturally be viewed as certain types of directed graphs, such as those introduced in 
\cite{TC1936,Stephen1996,Stephen1990,Stallings1983,SW1984}, and more recently in~\cite{AMT2023}. From this viewpoint, congruences of unary algebras could be replaced by congruences of the corresponding digraphs, and the results of Sections \ref{sect:left1}--\ref{sect:right} can all be reformulated in this way.
\end{rem}

\subsection{Sch\"{u}tzenberger groups}\label{subsect:H}

It is well known that the $\H$-classes of idempotents in a semigroup $S$ are precisely the maximal subgroups of $S$.
In fact, there is a way to generalise this and associate a group to \emph{every} $\H$-class. These groups are known as 
Sch\"{u}tzenberger groups, and we review them briefly here; for more information, see for example \cite[Section~2.4]{CPbook} or \cite[pp.~166--167]{Arbib1968}. 

Fix an $\H$-class $H$ of a semigroup~$S$.  Define the set
\[
T_R = T_R(H) = \set{t\in S^1}{Ht\sub H}=\set{t\in S^1}{Ht= H}=\set{t\in S^1}{Ht\cap H\neq\es}.
\]
The equality of the three sets above follows from 
Green's Lemma.  For $t\in T_R$, we define the map $\rho_t:H\to H$ by $x\rho_t=xt$ for all $x\in H$, and we set
\[
\Ga_R = \Ga_R(H) = \set{\rho_t}{t\in T_R}.
\]
\JEnew{(It is worth noting that the map $t\mt\rho_t$ is not necessarily injective:}
\NRrev{different elements of $T_R$ may well act on $H$ in the same way.)}
Then $\Ga_R$ is a group under composition, known as the (\emph{right}) \emph{Sch\"{u}tzenberger group of $H$}.  It is a simply transitive subgroup of the symmetric group~$\S_H$, meaning that for all $h,h'\in H$ there exists a unique $\al\in\Ga_R$ such that $h\al=h'$; it follows that $|\Ga_R|=|H|$.  The Sch\"{u}tzenberger groups associated to any two $\H$-classes within the same $\D$-class of $S$ are isomorphic.  If $H$ happens to be a group $\H$-class, then $\Ga_R(H)\cong H$; indeed, $\Ga_R$ is the (right) Cayley representation of $H$ in~$\S_H$ in this case.

Note that $\Ga_R$ was defined in terms of right translations.  
\JErev{There is also a \emph{left Sch\"{u}tzenberger group} $\Ga_L=\Ga_L(H)$, defined dually in terms of left translations, which are written to the left of their arguments.}
%
%
In fact, the groups $\Ga_R$ and $\Ga_L$ are isomorphic, with an explicit isomorphism given as follows.  Fix some $h_0\in H$.  For $\al\in\Ga_R$, let $\al'$ be the unique element of $\Ga_L$ for which $h_0\al = \al'h_0$.  Then $\Ga_R\to\Ga_L;\ \al\mt\al'$ is an isomorphism.  
 Note that the actions of $\Ga_R$ and $\Ga_L$ on $H$ commute, in the sense that $(\beta h)\alpha = \beta(h\alpha)$ for all $h\in H$, $\al\in\Ga_R$ and~$\be\in\Ga_L$.
Because of the isomorphism, we will often denote either of $\Ga_R(H)$ or $\Ga_L(H)$ simply by $\Ga(H)$, and call this \emph{the} Sch\"utzenberger group of $H$.

We can also define a group structure on $H$ itself, again with respect to some fixed $h_0\in H$.  For $h_1,h_2\in H$, let $\al_1,\al_2$ be the unique elements of $\Ga_R$ such that $h_1=h_0\al_1$ and $h_2=h_0\al_2$, and define $h_1\star h_2 = h_0\al_1\al_2$.  Note in fact that
\begin{equation}\label{eq:h1h2}
h_1\star h_2 = h_1\al_2 = \al_1'h_2,
\end{equation}
where $\Ga_R\to\Ga_L;\ \al\mt\al'$ is the isomorphism above.  Then $(H,\star)$ is a group with identity $h_0$, and is independent (up to isomorphism) of the choice of $h_0$.  Moreover, we have $(H,\star)\cong\Ga_R\cong\Ga_L$.  
If $H$ happens to be a group $\H$-class, then, taking $h_0$ to be the idempotent in $H$, we have $h_1\star h_2=h_1h_2$ for all $h_1,h_2\in H$.

We will need the following basic result later in the paper.  In the statement and proof, we keep the above notation, including the $\star$ operation on $H$, defined with respect to the fixed element~${h_0\in H}$.

\begin{lemma}\label{lem:sv}
Let $H$ be an $\H$-class of a semigroup $S$.  
For $s\in S^1$ and $h\in H$ the equality \NRrev{$us=u\star h$ (resp.~$su=h\star u$)} holds for some $u\in H$ if and only if it holds for all $u\in H$.
\end{lemma}

\pf
\NRrev{By symmetry, we only consider the equation $us=u\star h$, for which only the forwards implication needs proof.  So, suppose $us=u\star h$ for some $s\in S^1$ and $u,h\in H$, and let $v\in H$ be arbitrary.  
By \eqref{eq:h1h2} we have $u\star h=u\alpha$, where $\alpha$ is the unique element of $\Gamma_R$ with $h=h_0\alpha$.
Since $us\in H$, it follows that $s\in T_R$, and hence $\rho_s\in \Gamma_R$.
Now,  $u\alpha=u\star h =us=u\rho_s$.  
Since $\Gamma_R$ acts simply transitively on $H$ it follows that $\alpha=\rho_s$.
Using \eqref{eq:h1h2} once more we have
$vs=v\rho_s=v\alpha=v\star h$, as required.
}
\epf

\section{Left congruences, actions and $\L$-classes}\label{sect:left1}

We now apply the theory developed in Sections \ref{sect:Rees} and \ref{sect:unary} to the left congruences of a semigroup $S$.
Recall from Lemma \ref{la:SSS} that these are precisely the congruences of the left $S$-act ${}^S\!S$.
Theorem~\ref{thm:G} decomposes the height of $\LCong(S)=\Cong({}^S\!S)$ as the sum of the heights of the principal factors of $\G$-classes (when there are finitely many of these), which are now the $\L$-classes of $S$ (cf.~Lemma~\ref{la:SSS}).
Recall that the principal factor associated to an $\L$-class $L$ of $S$ is the left $S$-act
\[
\ol L = L\cup\{0\},
\qquad\text{with action}\qquad s\cdot x = \begin{cases}
sx &\text{if $x,sx\in L$,}\\
0 &\text{otherwise}.
\end{cases}
\]
Let us denote this $S$-act by ${}^S\ol L$.  
Acts of this type are well known in semigroup theory; for example, they appear as partial acts in \cite[Section 2.5]{Lallement1979}.

Combining Theorem \ref{thm:G} \JEnew{with both parts of} Lemma \ref{la:SSS} we obtain:

\begin{thm}\label{thm:S}
If $S$ is a semigroup with finitely many $\L$-classes, $L_1,\ldots,L_k$, then
\[
\Ht(\LCong(S)) = \sum_{r=1}^k \Ht(\Cong({}^S\ol L_r)) - k.\epfreseq
\]
\end{thm}

This reduces the computation of $\Ht(\LCong(S))$ to computing heights of principal factors of $\L$-classes of $S$.
A further reduction is given by noticing
 that  the $\L$-classes from the same  $\D$-class all make equal contributions, as a consequence of the following lemma.
 The result can be regarded as folklore; for instance see \cite[Proposition 2.5.2]{Lallement1979}, which deals with partial actions on $\R$-classes. We include a proof for completeness, and for later reference.  

\begin{lemma}\label{lem:D}
If $\L$-classes $L_1$ and $L_2$ of a semigroup $S$ belong to the same $\D$-class, then
${{}^S\ol L_1\cong {}^S\ol L_2}$.
\end{lemma}

\pf
By the assumption and Green's Lemma  there exists $a\in S^1$ such that
$L_1\to L_2;\ x\mt xa $ is an \JEnew{$\R$-class-preserving} bijection.
We  extend this to a bijection $f:{}^S\overline{L}_1\rightarrow {}^S\overline{L}_2$ by setting \JEnew{$0f:=0$}, and
we claim that $f$ is an $S$-act morphism, and hence an isomorphism.
To prove this, let $x\in {}^S \overline{L}_1$ and $s\in S$ be arbitrary.  
If $x=0$ then $(s\cdot x)f=0f=0=s\cdot 0=s\cdot (xf)$.
Otherwise $x\in L_1$, in which case \JErev{we have $sx\in L_1(=L_x)$ if and only if $sxa\in L_2(=L_{xa})$}, by Lemma \ref{lem:stab}\ref{stab1} with $y=xa$.
Now:
\bit
\item if $sx\in L_1$ and $sxa\in L_2$, then $s\cdot(x f) = s\cdot(xa) = sxa = (s\cdot x) f$;
\item if $sx\not\in L_1$ and $sxa\not\in L_2$, then $s\cdot (x f) = s\cdot (xa) = 0 = 0 f = (s\cdot x) f$.
\eit
Thus $f$ is indeed a morphism, and the proof is complete.
\epf

\begin{thm}\label{thm:S2}
Let $S$ be a semigroup with finitely many $\L$-classes,
let $L_1,\dots, L_n$ be representative $\L$-classes for the $\D$-classes of $S$, and let
$m_r$ be the number of $\L$-classes in the $\D$-class of~$L_r$. Then
\[
\Ht(\LCong(S)) = \sum_{r=1}^n m_r  (\Ht(\Cong({}^S\ol L_r)) - 1).
\]
\end{thm}

\pf
For each $1\leq r\leq n$, let the $\L$-classes in the $\D$-class of $L_r$ be $L_{r1},\ldots,L_{rm_r}$.  Then
\begin{align*}
\Ht(\LCong(S)) &= \sum_{r=1}^n\sum_{j=1}^{m_r} \Ht(\Cong({}^S\ol L_{rj})) - \sum_{r=1}^n m_r &&\text{by Theorem \ref{thm:S}}\\
&= \sum_{r=1}^n m_r \Ht(\Cong({}^S\ol L_r)) - \sum_{r=1}^n m_r &&\text{by Lemma \ref{lem:D}}\\
&= \sum_{r=1}^n m_r (\Ht(\Cong({}^S\ol L_r)) - 1). &&\qedhere
\end{align*}
\epf

\section{Left congruences and Sch\"{u}tzenberger groups}\label{sect:left2}

In this section we build on the previous section by focussing on the left $S$-act ${}^S\overline{L}$
associated to an arbitrary $\L$-class $L$ of $S$.
The $\H$-classes of $S$ contained in $L$ and their Sch\"{u}tzenberger groups will play an important role.
Thus, we will consider the following relation on $\overline{L}=L\cup\{0\}$:
\[
\H_{\ol L} := \H\restr_L\cup\{(0,0)\}.
\]

\begin{lemma}\label{lem:L}
If $L$ is an $\L$-class of a semigroup $S$, then $\H_{\ol L} \in \Cong({}^S\ol L)$.
\end{lemma}

\pf
It is clear that $\H_{\overline{L}}$ is an equivalence relation, so it remains to establish (left) compatibility.
To this end let $(u,v)\in\H_{\ol L}$ and $s\in S$. If $u=v=0$ then $(s\cdot u,s\cdot v)=(0,0)\in \H_{\ol L}$, so now suppose $u,v\in L$.
Lemma~\ref{lem:stab}\ref{stab1} gives $su\in L\iff sv\in L$. If $su,sv\not\in L$ then
$(s\cdot u,s\cdot v)=(0,0)\in \H_{\overline{L}}$. Otherwise, when $su,sv\in L$, Green's Lemma implies that the mapping
$x\mapsto sx$ is a bijection from $R_u$ to~$R_{su}$  that  preserves $\H$-classes. In particular,
$(s\cdot u,s\cdot v)=(su,sv)\in\H\restr_L\subseteq\H_{\overline{L}}$, completing the proof.
\epf

We now bring the Sch\"{u}tzenberger groups of the $\H$-classes into play. For an $\H$-class $H$, we will work with the representation $(H,\star)$ of $\Gamma(H)$, as introduced in Subsection \ref{subsect:H}. Recall that this representation depends on an arbitrary, but fixed, $h_0\in H$, which will always be assumed as given.

\begin{prop} \label{pr:Hrest}
Let $S$ be a semigroup, $L$ an $\L$-class of $S$, and $H$ an $\H$-class \JEnew{contained in}~$L$.
The mapping
\[
\phi:[\Delta_{\overline{L}},\H_{\overline{L}}]\rightarrow \LCong(H,\star);\ \sigma\mapsto\sigma\restr_{H}
\]
is a lattice isomorphism. Consequently, the interval $[\Delta_{\overline{L}},\H_{\overline{L}}]$ is isomorphic to the subgroup lattice 
$\Sub(\Gamma(H))$ of the Sch\"{u}tzenberger group of $H$.
\end{prop}

\pf
To prove that $\phi$ is well defined, let $\sigma\in [\Delta_{\overline{L}},\H_{\overline{L}}]$.
Certainly $\si\restr_H$ is an equivalence on $H$.  For (left) compatibility, suppose ${(h_1,h_2)\in\si\restr_H}$ and $a\in H$.  
Let $t\in S^1$ be such that $th_1=a\star h_1$.
By Lemma \ref{lem:sv} we also have $th_2=a\star h_2 $.
But then $(a\star h_1,a\star h_2) = (t\cdot h_1,t\cdot h_2) \in \si$, which shows that $\sigma\restr_H$ is left compatible.

We now show that 
\[
\JEnew{\sigma\subseteq\sigma' \iff \sigma\restr_H\subseteq\sigma'\restr_H \qquad\text{for all $\sigma,\sigma'\in  [\Delta_{\overline{L}},\H_{\overline{L}}]$,}}
\]
which implies that $\phi$ is an order-preserving injection.
 With the forwards implication being clear, suppose $\si\restr_H\sub\si'\restr_H$.  Let $(x,y)\in\si$; we must show that $(x,y)\in\si'$.  This is clear if $x=y=0$, so suppose instead that $x,y\in L$.  Since $\si\sub\H_{\ol L}$, we have $x,y\in H'$ for some $\H$-class $H'\sub L$.  By Green's Lemma we can fix $s,t\in S^1$ such that
\[
H\to H';\ u\mt su \AND H'\to H;\ v\mt tv
\]
are mutually inverse bijections.  Then $(tx,ty) = (t\cdot x,t\cdot y) \in \si\restr_H \sub \si'\restr_H$, and so
\[
(x,y) = (stx,sty) = (s\cdot tx,s\cdot ty) \in \si',
\]
as required.

We complete the proof by showing that $\phi$ is onto.
Let $\tau\in\LCong(H,\star)$, and consider $\tau^\sharp$, the congruence of $\overline{L}$ generated by $\tau$.
Certainly $\tau^\sharp\sub\H_{\ol L}^\sharp=\H_{\ol L}$.  To demonstrate that $\tau^\sharp\restr_H=\tau$, it suffices to establish the forwards inclusion, so suppose $(x,y)\in\tau^\sharp\restr_H$.  
\NRrev{Using the standard description of the left congruence given by a set of generating pairs (e.g.~see 
\cite[Lemma I.4.37]{KKM2000}),}
\JEnew{there is a sequence
$
x = x_0 , x_1 , \ldots , x_k = y
$
of elements of $\ol L$ such that for each $0\leq i<k$,}
\[
(x_i,x_{i+1}) = (s_i\cdot u_i,s_i\cdot v_i) \qquad\text{for some $s_i\in S^1$ and $(u_i,v_i)\in\tau$.}
\]
Since each $(x_i,x_{i+1}) \in \tau^\sharp \sub \H_{\ol L}$, and since $x_0=x\in H$, it follows that $x_i\in H$ for all $i$.  Note that $(s_i\cdot u_i,s_i\cdot v_i) = (s_iu_i,s_iv_i)$.  As $s_iu_i=x_i\in H$ and $(H,\star)$ is a group, there
exists $h_i\in H$ such that $s_iu_i=h_i\star u_i$. Lemma~\ref{lem:sv} tells us that  $s_iv_i=h_i\star v_i$.  
Combining this with the fact that $\tau$ is a left congruence on $(H,\star)$ and $(u_i,v_i)\in\tau$, it follows that
\[
(x_i,x_{i+1}) = (s_iu_i,s_iv_i) = (h_i\star u_i,h_i\star v_i) \in \tau.
\]
Hence $x = x_0 \mr\tau x_1 \mr\tau \cdots \mr\tau x_k = y$, as required.
\epf

The next result follows by combining \eqref{eq:Lxx} with Propositions \ref{pro:srs} and \ref{pr:Hrest} (and the Correspondence Theorem).

\begin{cor}
\label{co:LSGa}
Let $L$ be an $\L$-class of a semigroup $S$, let $H$ be an $\H$-class \JEnew{contained in} $L$, and let $\Gamma$ be the Sch\"{u}tzenberger group of $H$. Then
\[
\Ht(\Cong({}^S\overline{L}))\geq \Ht(\Sub(\Gamma))+\Ht(\Cong({}^S\overline{L}/\H_{\overline{L}}))-1.
\]
If $\H_{\overline{L}}$ is a modular element in $\Cong({}^S\overline{L})$ then
\[
\Ht(\Cong({}^S\overline{L}))= \Ht(\Sub(\Gamma))+\Ht(\Cong({}^S\overline{L}/\H_{\overline{L}}))-1.  \epfreseq
\]
\end{cor}

\JEnew{When $|L/\H| = m$ is finite, the act ${}^S\ol L/\H_{\ol L}$ has size $m+1$, so it follows that
\[
2\leq \Ht(\Cong({}^S\overline{L}/\H_{\overline{L}})) \leq m+1.
\]
\NRrev{Each of these bounds can be attained.}
For example, consider the case in which $S$ is a left zero semigroup of size $m$.  Then $L=S$ is itself a single $\L$-class, and since $L$ is $\H$-trivial we have ${}^S\ol L/\H_{\ol L} \cong {}^S\ol L$.  Moreover, every equivalence on ${}^S\ol L$ is a congruence, so the height of $\Cong({}^S\ol L)$, and hence of $\Cong({}^S\overline{L}/\H_{\overline{L}})$, 
is $m+1$.
The attainment of the lower bound $2$ is the subject of the discussion for the remainder of this section.}


\NRrev{We will say that an algebra $A$ is \emph{congruence-free}  if its only congruences are $\De_A$ and $\nab_A$.
This property is more commonly known as simplicity, but we will avoid this term, as it has a different meaning in semigroup theory.  A congruence $\sigma$ on $A$ is said to be \emph{maximal} if $\sigma\neq\nabla_A$ and there is no congruence $\tau\in\Cong(A)$ such that $\sigma\subsetneq\tau \subsetneq \nabla_A$.
By the Correspondence Theorem, this is equivalent to the quotient $A/\sigma$ being non-trivial and congruence-free,
and also to $\Cong(A/\sigma)$ having height $2$.}

\begin{prop}
\label{pr:sep}
For an $\L$-class $L$ of a semigroup $S$, the following are equivalent:
\begin{thmenumerate}
\item
\label{it:sep1}
The left $S$-act ${}^S\overline{L}/\H_{\overline{L}}$ is \NRrev{congruence-free.}
\item
\label{it:sep2}
$\H_{\overline{L}}$ is the unique maximal   congruence of ${}^S\overline{L}$.
\item
\label{it:sep3}
For all $x,y\in L$ with $(x,y)\not\in\H$ there exists $s\in S$ such that precisely one of  $sx$ or $sy$ belongs to $L$.
\end{thmenumerate}
When these conditions hold we have $\Cong({}^S\overline{L})=[\Delta_{\overline{L}},\H_{\overline{L}}]\cup\{\nabla_{\overline{L}}\}
\cong \Sub(\Gamma)^\top$, and hence
\[
\Ht(\Cong({}^S\overline{L}))= \Ht(\Sub(\Gamma))+1,
\]
where $\Ga$ is the Sch\"utzenberger group of any $\H$-class in $L$.
\end{prop}

\pf
\firstpfitem{\ref{it:sep2}$\implies$\ref{it:sep1}}  This is clear.

\pfitem{\ref{it:sep1}$\implies$\ref{it:sep3}}  Suppose $\H_{\ol L}$ is maximal, and let $x,y\in L$ with $(x,y)\not\in\H$.  By maximality, $\nab_{\ol L}$ is generated (as an $S$-act congruence) by $\H_{\ol L}\cup\{(x,y)\}$.  \JEnew{So there exists a sequence
$
{x = x_0 , x_1,\ldots, x_k = 0}
$
of elements of $\ol L$ such that for each $0\leq i<k$,}
\[
(x_i,x_{i+1}) = (s_i\cdot u_i,s_i\cdot v_i) \qquad\text{for some $s_i\in S^1$ and $(u_i,v_i)\in\H_{\ol L}\cup\{(x,y),(y,x)\}$.}
\]
Since $x_0\not=0$ and $x_k=0$, there exists $0\leq i<k$ for which $x_i\not=0$ and $x_{i+1}=0$.  So
\[
(x_i,0) = (x_i,x_{i+1}) = (s_i\cdot u_i,s_i\cdot v_i).
\]
We cannot have $(u_i,v_i)\in \H_{\ol L}$ (as $\H_{\ol L}$ is an $S$-act congruence and $(x_i,0)\not\in\H_{\ol L}$), so up to symmetry $(u_i,v_i)=(x,y)$.  Since $s_i\cdot x=x_i\in L$ and $s_i\cdot y=0$, it follows from the definition of the action that $s_ix\in L$ and $s_iy\not\in L$.

\pfitem{\ref{it:sep3}$\implies$\ref{it:sep2}}  It suffices to show that for any $x,y\in\ol L$ with $(x,y)\not\in\H_{\ol L}$, 
\NRrev{the congruence $(x,y)^\sharp$ generated by $(x,y)$ is equal to $ \nab_{\ol L}$.}
This follows immediately from Lemma \ref{lem:a0} if one of $x,y$ is $0$, so now suppose $x,y\in L$.  By assumption, and renaming $x,y$ if necessary, there exists $s\in S$ such that $sx\in L$ and $sy\not\in L$.  But then $(sx,0) = (s\cdot x,s\cdot y)\in(x,y)^\sharp$, and we again apply Lemma~\ref{lem:a0}.
\epf

\begin{defn}
\NRrev{We say an $\L$-class is $\H$-\emph{separable} if it satisfies one, and hence all, of the conditions \ref{it:sep1}--\ref{it:sep3} from Proposition \ref{pr:sep}.}
\end{defn}

We remark that if $L_1$ and $L_2$ are two $\L$-classes belonging to the same $\D$-class, then $L_1$ is $\H$-separable if and only if $L_2$ is $\H$-separable. This follows from the fact that the isomorphism ${}^S\overline{L}_1\rightarrow {}^S\overline{L}_2$  constructed in the proof of Lemma \ref{lem:D} preserves $\H$-classes by Green's Lemma.

\NRrev{
\JErev{It turns out that $\H$-separability of an $\L$-class $L$ is equivalent to a natural combinatorial condition when the $\D$-class containing $L$ is stable and regular, as will be the case in all of our applications.  To describe this condition, consider}
a $\D$-class $D$ of a semigroup $S$.  Suppose the~$\R$- and~$\L$-classes contained in $D$ are $R_i$ ($i\in I$) and $L_j$ ($j\in J)$, respectively.  The $\H$-classes in~$D$ are then $H_{ij}=R_i\cap L_j$ ($i\in I$, $j\in J$).  Define the $I\times J$ matrix
\[
M(D) = (m_{ij}), \WHERE m_{ij} = \begin{cases}
1 &\text{if $H_{ij}$ contains an idempotent,}\\
0 &\text{otherwise.}
\end{cases}
\]
Although this matrix depends on the chosen indexing of the $\R$- and $\L$-classes, we are interested in certain properties that are not dependent on the indexing.  

\begin{defn} 
\NRrev{We say a $\D$-class $D$ is \emph{row-faithful} (resp.\ \emph{column-faithful}) if no two rows (resp.\ columns) of $M(D)$ are equal.}
\end{defn}

Clearly, this notion is of interest only for $\D$-classes that contain idempotents, i.e. the regular $\D$-classes.
For those we have the following, \JErev{in which we must also assume stability:

\begin{prop}\label{prop:sep}
A stable, regular $\D$-class $D$ is row-faithful if and only if some (and hence every) $\L$-class in $D$ is $\H$-separable.
 \end{prop}

\pf
\firstpfitem{($\Rightarrow$)}
Suppose $D$ is row-faithful, and fix an $\L$-class $L\sub D$.  To establish $\H$-separability of $L$, we will verify that condition \ref{it:sep3} from Proposition \ref{pr:sep} holds.
Let $x,y\in L$ with ${(x,y)\not\in\H}$.  From $(x,y)\in\L$ it follows that $R_x\not=R_y$.  Since~$D$ is row-faithful, there exists an $\L$-class $L'\sub D$ such that (by symmetry)~$R_x\cap L'$ contains an idempotent, $e$ say, but $R_y\cap L'$ does not contain an idempotent.  Now, $ex = x\in L$, as $e\rR x$. Thus, the proof will be complete if we can show that $ey\not\in L$.  
This is certainly true if $ey\not\in D$, so suppose instead that $ey\in D$.  Then $ey\rJ e$, and it follows from stability that $ey\rR e$, whence $ey\in R_e$.  Since $L_e\cap R_y = L'\cap R_y$ contains no idempotent, it follows from \cite[Proposition 2.3.7]{Howie1995} that $ey\not\in R_e\cap L_y$.  But we showed above that $ey\in R_e$, so we must in fact have $ey\not\in L_y=L$.

\pfitem{($\Leftarrow$)}
Aiming to prove the contrapositive, suppose $D$ is not row-faithful. This means that there exist distinct $\R$-classes $R_1,R_2\sub D$ such that:
\begin{equation}
\label{eq:isorows}
R_1\cap L\cap E(S)=\es \iff R_2\cap L\cap E(S) = \es \qquad\text{for all } L\in D/\L.
\end{equation}
Since $D$ is regular, and since \eqref{eq:isorows} holds, there exists an $\L$-class $L\sub D$ such that both $R_1\cap L$ and $R_2\cap L$ contain  idempotents, say $e\in R_1\cap L$ and $f\in R_2\cap L$.  To complete the proof of this direction we will show that $L$ is not $\H$-separable, by showing that
\begin{equation}
\label{eq:LnotHsep}
se\in L \iff sf\in L \qquad\text{for every } s\in S.
\end{equation}
By symmetry, it suffices to prove only one direction of \eqref{eq:LnotHsep}.  So suppose $se\in L$.  Since $D$ is regular, $R_{se}$ must contain an idempotent, $g$ say.  By \cite[Theorem 2.3.4]{Howie1995}, the $\H$-class $R_1\cap L_g=R_e\cap L_g$ contains an inverse $t$ of $se$ such that $tse=e$ and $set=g$.  Since
\[
ts\leqJ t \mr\J e = tse \leqJ ts,
\]
it follows that $ts\mr\J t$, and then stability gives $ts\mr\R t$, i.e.~$ts\in R_1$.  Since $e$ is a left identity for its $\R$-class, we have  $ts = ets = (tse)ts = ts(ets) = (ts)^2$, so that $ts$ is an idempotent.  By~\eqref{eq:isorows}, the $\H$-class $R_2\cap L_{ts}$ also contains an idempotent, $h$ say.
Green's Lemma now implies the existence of some $u\in R_2\cap L_g$ such that $us=h$.
The idempotent~$h$ is a left identity for its $\R$-class, so that $f=hf=usf$.  The last conclusion says that $f\mr\L sf$, i.e.~that $sf\in L_f=L$, as required to complete the proof of 
\eqref{eq:LnotHsep}, and of the proposition.
\epf
}}

While $\H$-separability will prove invaluable in our concrete applications, its use in the special instance of a minimal $\L$-class $L$ is limited by the following observation.
In that case, $L=\overline{L}\setminus \{0\}$ is itself also an $S$-act, which we will denote by ${}^S\!L$.  It follows that $\nabla_L\cup\{(0,0)\}$ is a proper congruence
on $\overline{L}$ that contains $\H_{\overline{L}}$.
Hence~$\H_{\overline{L}}$ is maximal if and only if $L$ contains only one $\H$-class.
When a minimal $\L$-class $L$ is not an $\H$-class, we will need to compute $\Ht(\Cong({}^S\!L))$ by other means, e.g.~as in Corollary \ref{co:S3} below.  In preparation we note that Lemma \ref{lem:GS} applies here, yielding:

\begin{prop}
\label{pro:Lmin}
If $L$ is a minimal $\L$-class of a semigroup $S$, then
$\Cong({}^S\overline{L})\cong \Cong({}^S\!L)^\top$, and hence
\[
\Ht(\Cong({}^S\overline{L}))=\Ht(\Cong({}^S\!L))+1.  \epfreseq
\]
\end{prop}

Combining Theorem \ref{thm:S2} with Propositions \ref{pr:sep} and \ref{pro:Lmin} yields the following:

\begin{thm}\label{thm:S3}
Let $S$ be a semigroup with finitely many $\L$-classes, and
let $L_1,\dots, L_n$ be representative $\L$-classes for the $\D$-classes of $S$, with $L_1$ a minimal $\L$-class.
For each $r=1,\dots,n$, let $m_r$ be the number of $\L$-classes in the $\D$-class of $L_r$,  and let $\Gamma_r$ be the 
Sch\"{u}tzenberger group  of any $\H$-class of $L_r$. If $L_2,\dots, L_n$ are all $\H$-separable then
\[
\Ht(\LCong(S))=m_1 \Ht(\Cong({}^S\!L_1)) +\sum_{r=2}^n m_r\Ht(\Sub(\Gamma_r)).  \epfreseq
\]
\end{thm}


\JEnew{In all our applications, the value of $ \Ht(\Cong({}^S\!L_1))$ will be ascertained by noting
that either~$L_1$ is actually an $\H$-class (i.e.\ that ${}^S \overline{L}_1$ is also separable)} 
or that $\H_{L_1}$ is the unique maximal  congruence on ${}^S\!L_1$. 
We therefore record the resulting formulae in these two cases:


\begin{cor}
\label{co:S3}
Suppose that all the assumptions of Theorem \ref{thm:S3} are satisfied.
\begin{thmenumerate}
\item \label{it:S31}
If $L_1$ is an $\H$-class then
\[
\Ht(\LCong(S))=\sum_{r=1}^n m_r\Ht(\Sub(\Gamma_r)).
\]
\item  \label{it:S32}
If $\H_{L_1}$ is the unique maximal  congruence on ${}^S\!L_1$ then
\[
\Ht(\LCong(S))=m_1+\sum_{r=1}^n m_r\Ht(\Sub(\Gamma_r)).  \epfreseq
\]
\end{thmenumerate}
\end{cor} 

\section{Right congruences}\label{sect:right}

All the results of Sections \ref{sect:left1} and \ref{sect:left2} have left-right duals, concerning the right congruences of a semigroup.
For future reference we collect here the main conclusions, giving formulae for the height of the lattice of right congruences.

For an $\R$-class $R$ in a semigroup $S$, its
principal factor $\overline{R}{}^S$ is the right $S$-act
\[
\ol R = R\cup\{0\},
\qquad\text{with action}\qquad x\cdot s  = \begin{cases}
xs &\text{if $x,xs\in R$,}\\
0 &\text{otherwise}.
\end{cases}
\]
If $R$ is a minimal $\R$-class, then $R$ itself is a right $S$-act, denoted $R^S$, and this is a subact of~$\overline{R}{}^S$.

\begin{thm}[cf.~Theorem \ref{thm:S2}] \label{thm:Sr2}
Let $S$ be a semigroup with finitely many $\R$-classes,
let $R_1,\dots, R_n$ be representative $\R$-classes for the $\D$-classes of $S$, and let
$m_r$ be the number of $\R$-classes in the $\D$-class of $R_r$. Then
\[
\Ht(\RCong(S)) = \sum_{r=1}^n m_r \cdot (\Ht(\Cong(\ol R{}^S_r)) - 1).  \epfreseq
\]
\end{thm}

The property of $\H$-separability for an $\R$-class $R$ is to satisfy one, and hence all, of the conditions dual to~\ref{it:sep1}--\ref{it:sep3}
from Proposition \ref{pr:sep}. 
\JErev{Dually to Proposition \ref{prop:sep}, when the $\D$-class $D$ containing $R$ is stable and regular, $R$ is $\H$-separable if and only if $D$ is column-faithful.}

\begin{thm}[cf.~Theorem \ref{thm:S3}] \label{thm:Sr3}
Let $S$ be a semigroup with finitely many $\R$-classes, and
let $R_1,\dots, R_n$ be representative $\R$-classes for the $\D$-classes of $S$, with $R_1$ a minimal $\R$-class.
For each $r=1,\dots,n$, let $m_r$ be the number of $\R$-classes in the $\D$-class of $R_r$,  and let $\Gamma_r$ be the 
Sch\"{u}tzenberger group  of any $\H$-class of $R_r$. If $R_2,\dots, R_n$ are all $\H$-separable then
\[
\Ht(\RCong(S))=m_1 \Ht(\Cong(R_1^S)) +\sum_{r=2}^n m_r\Ht(\Sub(\Gamma_r)).  \epfreseq
\]
\end{thm}

\begin{cor}[cf.~Corollary \ref{co:S3}]
\label{co:Sr3}
Suppose that all the assumptions of Theorem \ref{thm:Sr3} are satisfied.
\begin{thmenumerate}
\item \label{it:Sr31}
If $R_1$ is an $\H$-class then
\[
\Ht(\RCong(S))=\sum_{r=1}^n m_r\Ht(\Sub(\Gamma_r)).
\]
\item \label{it:Sr32}
If $\H_{R_1}$ is the unique maximal  congruence on $R_1^S$ then
\[
\Ht(\RCong(S))=m_1+\sum_{r=1}^n m_r\Ht(\Sub(\Gamma_r)).  \epfreseq
\]
\end{thmenumerate}
\end{cor}

\section{Two-sided congruences}\label{sect:two}

We now prove analogous results concerning two-sided congruences.  These are not immediate corollaries of the results of Sections \ref{sect:left1}--\ref{sect:right}, although we can use some of those results in the proofs here.
The development will be broadly along the same lines as in Sections \ref{sect:left1} and \ref{sect:left2}.
Proofs are only given  if they are substantially different from the one-sided case.

For a $\J$-class $J$ in a semigroup $S$, its
principal factor ${}^S\overline{J}{}^S$ is the $(S,S)$-biact
\[
\ol J = J\cup\{0\},\qquad\text{with actions}\qquad
s\cdot x  = \begin{cases}
sx &\text{if $x,sx\in J$}\\
0 &\text{otherwise,}
\end{cases}
\qquad x\cdot s  = \begin{cases}
xs &\text{if $x,xs\in J$}\\
0 &\text{otherwise}.
\end{cases}
\]

\begin{thm}[cf.~Theorem \ref{thm:S}] \label{thm:SJ}
If $S$ is a semigroup with finitely many 
$\J$-classes, $J_1,\ldots,J_n$, then
\[
\Ht(\Cong(S)) = \sum_{r=1}^n \Ht(\Cong({}^S\ol J{}^S_r{})) - n.\epfreseq
\]
\end{thm}

Now consider an arbitrary $\L$-class $L$ contained in $J$.
Notice that ${}^S\overline{J}{}^S$ contains the underlying set $L\cup\{0\}$ of the left $S$-act ${}^S\overline{L}$.
However, this set may or may not be a left subact of ${}^S\overline{J}$ (the latter is the left $S$-act obtained by `forgetting' the right action on ${}^S\ol J{}^S$).  Similar comments apply to the right $S$-act $\ol R{}^S$.

\begin{prop}
\label{pro:stacts}
A $\J$-class $J$ of a semigroup $S$ is stable if and only if the following two dual  conditions are satisfied:
\begin{thmenumerate}
\item \label{stacts1}
${}^S\overline{L}$ is a subact of  ${}^S\overline{J}$ for every $\L$-class $L\subseteq J$; and 
\item \label{stacts2}
$\overline{R}{}^S$ is a subact of  $\overline{J}{}^S$ for every $\R$-class $R\subseteq J$.
\end{thmenumerate}
\end{prop}

\begin{proof}
\firstpfitem{($\Rightarrow$)}
To show \ref{stacts1},
let $x\in L$ and $s\in S$. If $sx\in L$ then $s\cdot x=sx$ in both ${}^S\overline{L}$ and ${}^S\overline{J}$.
Otherwise, if $sx\not\in L$, then $sx\not\in J$ by stability, and hence $s\cdot x=0$ in both acts.
Condition~\ref{stacts2} is verified analogously.

\pfitem{($\Leftarrow$)}
Let $x\in J$ and $s\in S$. We will show that $sx\in J$ implies $sx\in L$, where $L$ is the $\L$-class of~$x$.
\NRrev{Suppose $sx\in J$. So $s\cdot x=sx$ in ${}^S\overline{J}$. But then $s\cdot x=sx$ in ${}^S\overline{L}$, i.e.~$sx\in L$.
The other half of the definition of stability is checked analogously.}
\end{proof}

Thus, in the presence of stability, given $s\in S$ and $x\in \overline{J}$ we can unambiguously write $s\cdot x$
 for the result of the action of $s$ on $x$ in ${}^S\overline{J}{}^S$ and in ${}^S\overline{L}_x$, and similarly for $x\cdot s$.

In what follows we collect some results concerning the individual acts ${}^S\overline{J}{}^S$.
For this we will need to additionally assume that $J$ is stable.
Throughout  $\H_{\ol J}$ will stand for the relation $\H\restr_J\cup\{(0,0)\}$.

\begin{lemma}[cf.~Lemma \ref{lem:L}]\label{lem:J}
If $J$ is a stable $\J$-class of a semigroup $S$, then $\H_{\ol J} \in \Cong({}^S\ol J{}^S)$.
\end{lemma}

\pf
The relation $\H_{\ol J}$ is clearly an equivalence.
It equals the union of all $\H_{\ol L}$, with $L$ an $\L$-class contained in $J$.
By Lemma \ref{lem:L} each $\H_{\ol L}$ is left-compatible, and hence, using Proposition~\ref{pro:stacts},~$\H_{\ol J}$ is left-compatible.
The proof of right-compatibility is dual.
\epf

\begin{prop}[cf.~Proposition \ref{pr:Hrest}] \label{pr:JHrest}
Let $S$ be a semigroup, $J$ a stable $\J$-class of $S$, and $H$ an $\H$-class \JEnew{contained in} $J$.
The mapping
\[
\phi:[\Delta_{\overline{J}},\H_{\overline{J}}]\rightarrow \Cong(H,\star);\ \sigma\mapsto\sigma\restr_{H}
\]
is a lattice isomorphism. Thus the interval $[\Delta_{\overline{J}},\H_{\overline{J}}]$ is isomorphic to the normal subgroup lattice 
$\NSub(\Gamma(H))$ of the Sch\"{u}tzenberger group of $H$.
\end{prop}

\pf
To prove that $\phi$ is well defined, let $\sigma\in [\Delta_{\overline{J}},\H_{\overline{J}}]$.
By Proposition \ref{pro:stacts}, the relation $\si\restr_{\ol L}$ is a congruence on ${}^S\overline{L}$, where $L$ is the $\L$-class containing $H$.
Proposition \ref{pr:Hrest} then implies that $\si\restr_H$ is a left congruence on $(H,\star)$.
Dually, $\si\restr_H$ is a right congruence on $(H,\star)$, \NRrev{hence} it is a congruence.

Continuing to follow the proof of Proposition \ref{pr:Hrest}, we show that $\sigma\subseteq\sigma' \iff \sigma\restr_H\subseteq\sigma'\restr_H$ for all $\sigma,\sigma'\in  [\Delta_{\overline{J}},\H_{\overline{J}}]$.  With only the reverse implication requiring a proof,
suppose $\si\restr_H\sub\si'\restr_H$.  
 Since both $\sigma,\sigma'$ are contained in $\H_{\overline{J}}$ it is sufficient to show
 $\sigma\restr_{H'}\subseteq\sigma'\restr_{H'}$ where $H'$ is an arbitrary $\H$-class in $\J$.
 Let $L$ be the $\L$-class containing $H$, and let $R$ be the $\R$-class containing $H'$.
 Since~$J$ is stable, it is in fact a $\D$-class, and hence $H'':=L\cap R$ is an $\H$-class.
 By Proposition~\ref{pr:Hrest}, we have $\sigma\restr_{\overline{L}}\subseteq \sigma'\restr_{\overline{L}}$, and hence
 in particular $\sigma\restr_{H''}\subseteq \sigma'\restr_{H''}$.
 Now we use the dual of Proposition~\ref{pr:Hrest} (with respect to $H''\subseteq R$), to deduce that
 $\sigma\restr_{\overline{R}}\subseteq \sigma'\restr_{\overline{R}}$, and hence $\sigma\restr_{H'}\subseteq \sigma'\restr_{H'}$,
 as required.

The proof is completed by showing that $\phi$ is onto, for which it suffices to show that  ${\tau^\sharp\restr_H\subseteq\tau}$
for every $\tau\in\Cong(H,\star)$, where $\tau^\sharp$ is the congruence of $\overline{J}$ generated by $\tau$.
For any ${(x,y)\in \tau^\sharp\restr_H}$ \JEnew{we have a sequence
$
x = x_0 , x_1 ,\ldots, x_k = y
$
of elements of $\ol J$ such that for each $0\leq i<k$,}
\[
(x_i,x_{i+1}) = (s_i\cdot u_i\cdot t_i,s_i\cdot v_i\cdot t_i) \qquad\text{for some $s_i,t_i\in S^1$ and $(u_i,v_i)\in\tau$.}
\]
Since each $(x_i,x_{i+1}) \in \tau^\sharp \sub \H_{\ol J}$, and since $x_0 = x\in H$, we have $x_i\in H$ for all $i$.  Note that $(s_i\cdot u_i\cdot t_i,s_i\cdot v_i\cdot t_i) = (s_iu_it_i,s_iv_it_i)$.  
From $u_i,s_iu_i\in J$ and stability it follows that $s_i u_i\mr\L u_i$.
Similarly, from $s_iu_i,s_i u_i t_i \in J$ we have $s_i u_i \mr\R s_i u_i t_i$. But $u_i,s_iu_it_i\in H$, and hence we have $s_iu_i\in H$ as well.
As in the proof of Proposition \ref{pr:Hrest}, we now see that
there
exists $h_i\in H$ such that $s_iu_i=h_i\star u_i$ and  $s_iv_i=h_i\star v_i$.  
Since $s_iu_i,s_iv_i\in H$, a dual argument shows that there exists $h_i'\in H$ such that
$s_iu_it_i=(s_i u_i)\star h_i'$ and  $s_iv_it_i=(s_i v_i)\star h_i'$.
Combining, we obtain 
\[
(x_i,x_{i+1}) = (s_iu_it_i,s_iv_it_i) = (h_i\star u_i\star h_i',h_i\star v_i\star h_i') \in \tau,
\]
from which $x = x_0 \mr\tau x_1 \mr\tau \cdots \mr\tau x_k = y$, as required.
\epf

\begin{cor}[cf.~Corollary \ref{co:LSGa}]
\label{co:JSGa}
Let $J$ be a stable $\J$-class of a semigroup $S$, let $H$ be an $\H$-class in $J$, and let $\Gamma$ be the Sch\"{u}tzenberger group of $H$. Then
\[
\Ht(\Cong({}^S\overline{J}{}^S))\geq \Ht(\NSub(\Gamma))+\Ht(\Cong({}^S\overline{J}{}^S/\H_{\overline{J}}))-1.
\]
If $\H_{\overline{J}}$ is a modular element in $\Cong({}^S\overline{J}{}^S)$ then
\[
\Ht(\Cong({}^S\overline{J}{}^S))= \Ht(\NSub(\Gamma))+\Ht(\Cong({}^S\overline{J}{}^S/\H_{\overline{J}}))-1.
\epfreseq
\]
\end{cor}

Again, $\H$-separability is one way to achieve modularity of $\H_{\overline{J}}$, and this time we have the following:

\begin{prop}[cf.~Proposition \ref{pr:sep}]
\label{pr:Jsep}
\JEnew{For a stable $\J$-class $J$ of a semigroup $S$, the following are equivalent:}
\begin{thmenumerate}
\item
\label{it:Jsep1}
The  $(S,S)$-biact ${}^S\overline{J}{}^S/\H_{\overline{J}}$ is congruence-free. 
\item
\label{it:Jsep2}
$\H_{\overline{J}}$ is the unique maximal  congruence of ${}^S\overline{J}{}^S$.
\item
\label{it:Jsep3}
For all $x,y\in J$ with $(x,y)\not\in\H$ there exist $s,t\in S^1$ such that precisely one of  $sxt$ or $syt$ belongs to $J$.
\item
\label{it:Jsep4}
All $\L$-classes and all $\R$-classes \JEnew{contained in} $J$ are $\H$-separable.
\end{thmenumerate}
When these conditions hold we have $\Cong({}^S\overline{J}{}^S)=[\Delta_{\overline{J}},\H_{\overline{J}}]\cup\{\nabla_{\overline{J}}\} \cong\NSub(\Ga)^\top$, and hence
\[
\Ht(\Cong({}^S\overline{J}{}^S))= \Ht(\NSub(\Gamma))+1,
\]
where $\Ga$ is the Sch\"utzenberger group of any $\H$-class in $J$.
\end{prop}

\pf
The equivalence of \ref{it:Jsep1}, \ref{it:Jsep2} and \ref{it:Jsep3} is analogous to the proof of  Proposition \ref{pr:sep}.

\pfitem{\ref{it:Jsep3}$\implies$\ref{it:Jsep4}} 
It is sufficient to show that an arbitrary $\L$-class $L\subseteq J$ is $\H$-separable, the proof for $\R$-classes being dual.
Let $x,y\in L$ with $(x,y)\not\in\H$. By \ref{it:Jsep3}, and renaming if necessary, there exist $s,t\in S^1$ such that  $sxt\in J$ and $syt\not\in J$.
From $x,sxt\in J$ it follows that $sx,xt\in J(=J_x)$, and then stability implies
$sx\in L$ and $xt\mr\R x$.
By the dual of Lemma \ref{lem:stab} we have $yt\mr\R y$.
Since~$\R$ is a left congruence, it follows that $syt\mr\R sy$, implying $sy\not\in J$, and hence $sy\not\in L$, as required.

\pfitem{\ref{it:Jsep4}$\implies$\ref{it:Jsep3}} 
Let $x,y\in J$ with $(x,y)\not\in\H$.
Without loss \NRrev{of generality} assume that $(x,y)\not\in\R$.
Being stable,~$J$ is a $\D$-class, so there exists $z\in J$ with $x\mr\L z\mr\R y$.
By assumption, the $\L$-class $L:=L_x=L_z$ is $\H$-separable.  As $(x,z)\not\in\H$ there exists $s\in S$ such that
precisely one of $sx,sz$ is in $L$; it then follows from stability that precisely one of $sx,sz$ is in $J$. Since $z\mr\R y$, it follows from Lemma~\ref{lem:stab}\ref{stab2} that \NRrev{$sz\in J$ if and only if $sy\in J$, completing the proof of this part and of the proposition.}
\epf

\begin{defn}
\JEnew{We say a stable $\J$-class is $\H$-\emph{separable} if it satisfies one, and hence all, of the conditions \ref{it:Jsep1}--\ref{it:Jsep4} of Proposition \ref{pr:Jsep}.}
\end{defn}

\JErev{When the stable $\J$-class is also regular, we again obtain a combinatorial characterisation of $\H$-stability, as follows by combining condition \ref{it:Jsep4} of Proposition \ref{pr:Jsep} with Proposition \ref{prop:sep} and its dual:

\begin{prop}
A stable, regular $\J$-class $J$ is $\H$-separable if and only if the matrix $M(J)$ is both row- and column-faithful.  \epfres
\end{prop}}



Let us now consider a minimal $\J$-class $J$ in a semigroup $S$, and assume that it is stable.
Notice that, unlike its one-sided counterparts, $J$ is necessarily unique, and is then the (unique) minimal ideal of $S$.
As with minimal $\L$- and $\R$-classes, $J$ is separable if and only if it is a group, which does not cover all our intended applications.
Also analogous to the one-sided case, 
$J=\overline{J}\setminus \{0\}$ is an $(S,S)$-biact, denoted ${}^S\!J^S$, and we have the following:

\begin{prop}[cf.~Proposition \ref{pro:Lmin}]
\label{pro:Jmin}
If $J$ is a stable minimal $\J$-class of a semigroup $S$ then
$\Cong({}^S\overline{J}{}^S)\cong \Cong({}^S\!J{}^S)^\top$, and hence
\[
\Ht(\Cong({}^S\overline{J}{}^S))=\Ht(\Cong({}^S\!J{}^S))+1.  \epfreseq
\]
\end{prop}

\begin{thm}[cf.~Theorem \ref{thm:S3}]
\label{thm:JS3}
Let $S$ be a stable semigroup with finitely many $\J$-classes, $J_1,\dots, J_n$, with $J_1$ minimal.
For each $r=1,\dots,n$, let $\Gamma_r$ be the 
Sch\"{u}tzenberger group  of any $\H$-class of $J_r$. If $J_2,\dots, J_n$ are all $\H$-separable then
\[
\Ht(\Cong(S))= \Ht(\Cong({}^S\!J_1^S)) +\sum_{r=2}^n \Ht(\NSub(\Gamma_r)). \epfreseq
\]
\end{thm}

As we said in the preamble to Corollary \ref{co:S3}, in all our applications it will be the case that 
either $\H_L=\nabla_L$ or $\H_L$ is the unique maximal  congruence of ${}^S\!L$, where $L$ is a minimal $\L$-class of~$S$.
Analogous comments hold for a minimal $\R$-class. We combine these two conditions to obtain the following result concerning ${}^S\!J^S$ where $J$ is a stable minimal $\J$-class. In the statement we use the notation $\H_J=\H\restr_J$, $\L_J=\L\restr_J$ and $\R_J=\R\restr_J$, which are congruences on ${}^S\!J^S$ by \cite[Lemma~3.12]{EMRT2018}.

\NRrev{
\begin{lemma}
\label{la:diam}
Let $J$ be a stable minimal $\J$-class in a semigroup $S$.
Suppose that for some (equivalently, every) $\L$-class $L\subseteq J$ one of the following two conditions hold:
\begin{itemize}
\item
$\H_L$ is the unique maximal  congruence of \JEnew{${}^S\!L$}; or 
\item
$\H_L=\nabla_L$.
\end{itemize}
Also suppose that, dually, for some (equivalently, every)
$\R$-class $R\subseteq J$ one of the following two conditions hold:
\begin{itemize}
\item
$\H_R$ is the unique maximal  congruence of \JEnew{$R^S$}; or
\item
$\H_R=\nabla_R$.
\end{itemize}
Then
\[
\Cong({}^S\!J^S)=[\Delta_J,\H_J]\cup\{\L_J,\R_J,\nabla_J\}.
\]
\end{lemma}}


\begin{proof}
It suffices to prove that for any $x,y\in J$ with $(x,y)\not\in \H$ we have $(x,y)^\sharp\in \{\L_J,\R_J,\nabla_J\}$,
where $(x,y)^\sharp$ denotes the congruence of ${}^S\!J^S$ generated by $(x,y)$.
We distinguish three cases depending on whether or not $x,y$ are $\L$- or $\R$-related.

\vspace{2mm}
\noindent\textit{Case 1: $(x,y)\not\in\L$ and $(x,y)\in\R$.}
We claim that $(x,y)^\sharp=\R_J$, with the direct inclusion clear.
Let $R=R_x=R_y$.
By assumption, the congruence on $R^S$ generated by $(x,y)$ is $\nabla_R$.
\NRrev{The act $R^S$ is a subact of $J^S$ by Proposition \ref{pro:stacts} and stability, and hence $\nabla_R\subseteq (x,y)^\sharp$.}
Translating this by left multiplication yields $\nab_{R'}\sub(x,y)^\sharp$ for every $\R$-class $R'\sub J$, and so $\R_J\subseteq (x,y)^\sharp$, completing the proof of the claim.

\vspace{2mm}
\noindent\textit{Case 2: $(x,y)\in\L$ and $(x,y)\not\in\R$.} This is dual to Case 1, and here we have $(x,y)^\sharp = \L_J$.

\vspace{2mm}
\noindent\textit{Case 3: $(x,y)\not\in\L$ and $(x,y)\not\in\R$.}
By minimality of $J$ we have $xy\in J$, and then stability implies $x\mr\R xy\mr\L y$.
Similarly, $x\mr\H x^2$.
Hence, the pair $(x^2,xy)=(x\cdot x,x\cdot y)$ belongs to $(x,y)^\sharp$ and to $\R$, but not to $\L$.
Case 1 implies $\R_J=(x^2,xy)^\sharp\subseteq (x,y)^\sharp$.
Dually, $\L_J\subseteq (x,y)^\sharp$, and hence $\nabla_J=\L_J\vee \R_J\subseteq (x,y)^\sharp$, completing the proof.
\end{proof}

Figure \ref{fig:SJS} shows the congruence lattice of ${}^S\!J^S$ in the case that $J$ satisfies the conditions of Lemma \ref{la:diam}, but we note that the congruences $\H_J,\L_J,\R_J,\nabla_J$ are not always distinct.
When precisely one of $\H_L=\nabla_L$ or $\H_R=\nabla_R$ holds, the lattice degenerates to ${\Cong({}^S\!J^S)\cong \NSub(\Gamma)^\top}$.
Thus, we obtain:

\begin{cor}[cf.~Corollary \ref{co:S3}]
\label{co:JS3}
Suppose that all the assumptions of Theorem \ref{thm:JS3} are satisfied.
Also suppose that the minimal $\J$-class $J_1$ satisfies the assumptions of Lemma \ref{la:diam},
and fix an $\L$-class $L\subseteq J$ and an $\R$-class $R\subseteq J$.
Then:
\begin{equation}\label{eq:HtSJS}
\Ht(\Cong(S)) = \sum_{r=1}^n \Ht(\NSub(\Gamma_r))+ \begin{cases}
0 &\text{if } \H_L=\nabla_L \text{ and } \H_R=\nabla_R\\
1 &\text{if } \H_L=\nabla_L \text{ and } \H_R\neq\nabla_R\\
1  &\text{if } \H_L\neq\nabla_L \text{ and } \H_R=\nabla_R\\
2 &\text{if } \H_L\neq\nabla_L \text{ and } \H_R\neq\nabla_R.
\end{cases}
\end{equation}
~\\[-10mm]\epfres
\end{cor}

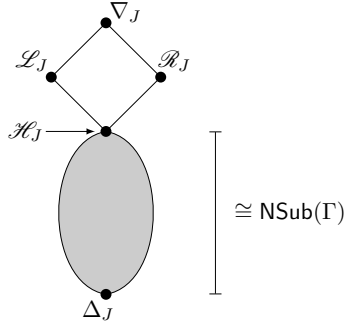
\begin{figure}
\begin{center}
\scalebox{0.8}{
\begin{tikzpicture}[scale=0.9]
\draw[fill=black!20] (0,4) to [bend left=80] (0,7) to [bend left=80] (0,4);
\draw (0,7) -- (-1,8)--(0,9)--(1,8)--(0,7);
\foreach \x/\y in {0/4,0/7,-1/8,1/8,0/9} {\fill (\x,\y)circle(.1); }
\foreach \x/\y/\z in {
0.3/3.7/\De_J,
-.9/8.3/\L_J,
1.7/8.3/\R_J,
0.8/9.2/\nab_J,
-1/7/\H_J
} {\node[left] () at (\x,\y) {$\z$};}
\draw[-{latex}] (-1.1,7)--(-.2,7);
\draw[|-|] (2,7)--(2,4);
\node[right] () at (2.2,5.5) {$\cong\NSub(\Ga)$};
\node[left,white] () at (-2.2,5.5) {$\cong\NSub(\Ga)$};
\end{tikzpicture}
}
\caption{The lattice $\Cong({}^S\!J^S)$ in the situation described in Lemma \ref{la:diam}.}
\label{fig:SJS}
\end{center}
\end{figure}

\section{Applications}\label{sect:applications}

As \NRrev{the culmination} of the paper, we now apply the theoretical results above to give explicit formulae for the heights of the lattices of left, right and two-sided congruences of several `classical' semigroups of transformations, matrices and partitions.
These formulae are recorded in Table~\ref{tab:formulae}, and some computed values are given in Table \ref{tab:small}. 
Following the formulae given in Corollaries~\ref{co:S3},~\ref{co:Sr3} and~\ref{co:JS3}, in order to compute 
$\Ht(\LCong(S))$, $\Ht(\RCong(S))$ and $\Ht(\Cong(S))$ for a given semigroup $S$ we need to do the following:
\begin{itemize}
\item
List the $\D$-classes $D_1,\dots,D_n$ of $S$, and identify the minimal $\D$-class $D_1$.
\item
For each $\D$-class $D_r$, find the number $|D_r/\L|$ of its $\L$-classes and the number $|D_r/\R|$ of its $\R$-classes.
\item
For each $\D$-class $D_r$, identify the Sch\"{u}tzenberger group $\Gamma_r$ of any of its $\H$-classes, and determine the heights of the 
lattices of subgroups and normal subgroups of $\Gamma_r$.
\item
For each non-minimal $D_r$ confirm that it is $\H$-separable. In all our examples this can be done by checking that the matrix $M(D_r)$ is row- and column-faithful.
\item
For the minimal $\D$-class $D_1$, fix an $\L$-class $L\subseteq D_1$ and an $\R$-class $R\subseteq D_1$.
Check if~$L$ is $\H$-separable (which is the case precisely when $|D_1/\R|=1$) or if $\H_{L}$ is the unique maximal  congruence
of ${}^S\!L_1$. This last step, when needed, requires ad-hoc arguments, which, fortunately, are very easy in our examples or can be deduced from known results in the literature. Perform the dual check for $R$.
\end{itemize}
Here is the list of semigroups featuring in Table \ref{tab:formulae} \NRrev{(more details will be given below)}:
\begin{itemize}
\item
the full transformation monoid $\T_n$, \NRrev{consisting of all mappings $[n]\rightarrow [n]$, where $[n]:=\{1,\ldots,n\}$};
\item
the partial transformation monoid $\PT_n$, consisting of all partial mappings $[n]\rightarrow [n]$;
\item
the symmetric inverse monoid $\I_n$, consisting of all partial bijections $[n]\rightarrow [n]$;
\item
the monoid $\O_n$ of all order-preserving mappings $[n]\rightarrow [n]$;
\item
the general linear monoid $\M(n,q)$, consisting of all $n\times n$ matrices over the finite field $\FF_q$;
\item
the partition monoid $\P_n$, \NRrev{consisting of all set partitions on the set \JEnew{$[n]\cup[n]'$, where $[n]' := \{1',\dots,n'\}$ (e.g.\ see \cite{EG2017})};}
\item
the Brauer monoid $\B_n$, consisting of all partitions from $\P_n$ with all blocks of size $2$;
\item
the Temperley--Lieb monoid $\TL_n$, consisting of all planar Brauer partitions;
\item
the dual symmetric inverse monoid $\I_n^*$,  consisting of all block bijections $[n]\to[n]$,
 i.e.~partitions from $\P_n$ \JEnew{in which every block is a \emph{transversal}, i.e.~contains points from both $[n]$ and~$[n]'$}. 
\end{itemize}
The same methodology can be followed to obtain analogous formulae for essentially all natural \NRrev{finite} semigroups of transformations or partitions that are commonly studied in the literature.

\NRrev{Before we look in detail at our selected examples, we make a remark about terminology.
All the semigroups in the above list are known to be regular, which, among other things, means that  every $\L$-class and every $\R$-class contain an idempotent. 
As mentioned in Subsection \ref{subsect:H}, the $\H$-class $H$ of an idempotent is a maximal subgroup, and is isomorphic to the Sch\"{u}tzenberger group of $H$. Thus, in what follows we could talk about maximal subgroups instead of 
Sch\"{u}tzenberger groups, which would be more in line with the literature on these specific semigroups. However, we have opted to retain Sch\"{u}tzenberger groups for the sake of internal consistency of this article.}

\subsection{The full transformation monoid $\T_n$}

The full transformation monoid $\T_n$ consists of all mappings $[n]\rightarrow [n]$ under the composition of mappings.
\NRrev{The mappings in $\T_n$ are written to the right of their arguments, and the composition is from left to right, i.e.~$x(\alpha\beta)=(x\alpha)\beta$.}
\JEnew{The order of $\T_n$} is $n^n$.  In the following discussion we assume that $n\geq2$.  Green's relations
 \JErev{$\D(=\!\!\J)$}, $\L$ and $\R$ are governed by ranks, images and kernels of mappings:
\[
\alpha\mr\D\beta \iff \rank(\alpha)=\rank(\beta),\quad
\alpha\mr\L\beta \iff \im(\alpha)=\im(\beta),\quad
\alpha\mr\R\beta \iff \ker(\alpha)=\ker(\beta).
\]
See for example \cite[Section 2.2]{CPbook}.
In particular, $\T_n$ has $n$ $\D$-classes, denoted $D_1,\dots,D_n$,
where $D_r=\set{\alpha\in \T_n}{ \rank(\alpha)=r}$. The minimal $\D$-class is $D_1$.
The $\L$-classes in $D_r$ are in one-one correspondence with possible images of size $r$, and hence
 $|D_r/\L|=\binom{n}{r}$. Similarly, the $\R$-classes in $D_r$ correspond to possible kernels with $r$ blocks, and so
 $|D_r/\R|=S(n,r)$, the Stirling number of the second kind. 
 The Sch\"{u}tzenberger groups of $\H$-classes in $D_r$ are all isomorphic to the symmetric group $\S_r$.
 Each non-minimal $\D$-class $D_r$ ($2\leq r\leq n$) is $\H$-separable. This can be verified by checking that the matrix $M(D_r)$ is row- and column-faithful; alternatively see~\cite{ER2023}.
 The minimal $\D$-class $D_1$ is also an $\R$-class, and contains $n$ singleton $\L$-classes.

 It now follows that each $\L$-class in $D_1$ is separable, and we obtain a formula for the height of $\LCong(\T_n)$ 
 by applying Corollary \ref{co:S3}\ref{it:S31}: 
 \[
 \Ht(\LCong(\T_n)) = \sum_{r=1}^n \binom nr \Ht(\Sub(\S_r)).
 \]
 However, $R=D_1$, the unique $\R$-class in $D_1$ \NRrev{is not an $\H$-class} (for $n\geq2$). 
 Hence we would like to apply Corollary \ref{co:Sr3}\ref{it:Sr32}, for which we need to check that
 $\H_R$ is the unique maximal  congruence on $R^{\T_n}$.
 To this end, let $\alpha,\beta\in R$ be any two distinct constant mappings, say with $\im(\alpha)=\{x\}$ and $\im(\beta)=\{y\}$, where $x\neq y$.
 Further, let $\gamma,\delta\in R$ be any constant mappings, say with the images $u$ and $v$, respectively.
 Let $\zeta\in \T_n$ be any mapping satisfying 
 $x\zeta =u$, $y\zeta =v$. Then $(\gamma,\delta)=(\alpha\cdot \zeta,\beta\cdot\zeta)\in(\alpha,\beta)^\sharp$. In other words, $(\alpha,\beta)^\sharp=\nabla_R$, proving the claim. So, applying Corollary \ref{co:Sr3}\ref{it:Sr32}, we obtain
 \[
 \Ht(\RCong(\T_n)) = 1+\sum_{r=1}^n S(n,r) \Ht(\Sub(\S_r)).
\]
Finally, for the height of $\Cong(\T_n)$, we apply Corollary \ref{co:JS3}.
From the discussion above, we have~${\H_L=\nabla_L}$ for any $\L$-class $L\sub D_1$, and $\H_R$ is the unique maximal  congruence on~$R^{\T_n}$ (where ${R=D_1}$).  Consequently, the extra additive contribution in \eqref{eq:HtSJS} is $1$ (for $n\geq2$) and we have
$\Ht(\Cong(\T_n))=1+\sum_{r=1}^n\Ht(\NSub(\S_r))$.
Now, it is known that 
\begin{equation} \label{eq:htnsr}
(\Ht(\NSub(\S_r)))_{\JErev{r\geq 1}}= (\JErev{1},2,3,4,3,3,3,\dots),
\end{equation}
and it follows that
\[
\Ht(\Cong(\T_n))=3n-1 \quad \text{for } n\geq 4.
\]
The same result can, of course, be obtained by recalling Malcev's description of congruences on~$\T_n$~\cite{Malcev1952}: they in fact form a chain
of length $3n-1$ (for $n\geq 4$).
By contrast, the lattices of one-sided congruences of $\T_n$ are far from being understood, and have very complicated structure even for small $n$; \JErev{see Figure \ref{fig:LT3} for the lattice $\LCong(\T_3)$, produced using GAP \cite{GAP4,Semigroups}.}

\begin{figure}[t]
\begin{center}
\includegraphics[width=0.9\textwidth]{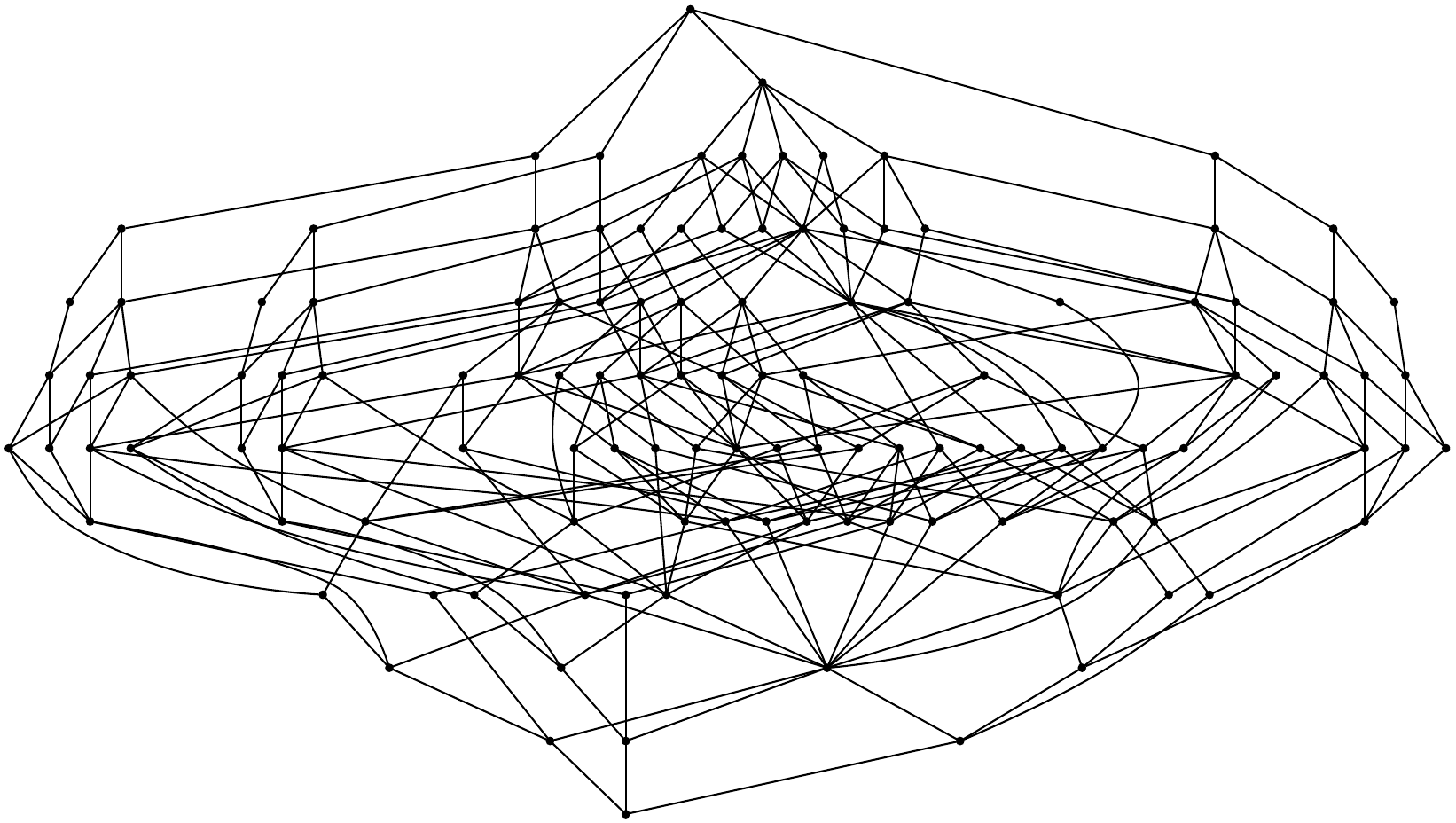}
\caption{The lattice $\LCong(\T_3)$ of left congruences of the full transformation monoid $\T_3$.}
\label{fig:LT3}
\end{center}
\end{figure}


\subsection{The partition monoid $\P_n$}

We now follow the same trajectory for the partition monoid $\P_n$, but somewhat more rapidly, again assuming that $n\geq2$.
For the definition of the product in $\P_n$, see for example \cite{EG2017}, which also includes the following information regarding Green's relations.
The $\D$-classes $D_r$ of $\P_n$ again consist of partitions of a fixed rank, but this time the available ranks are $r=0,1,\dots,n$, with~$D_0$ minimal.
The $\L$- and $\R$-classes in $D_r$ correspond to partitions of $[n]$ with $r$ marked blocks,
\NRrev{where markings are used to indicate which blocks participate in transversals.
Consequently, the number of $\R$- and $\L$-classes in $D_r$ is
$\sum_{k=r}^{n} S(n,k)\binom{k}{r}$.}
The Sch\"{u}tzenberger groups are again $\S_r$.
That each $D_r$ is $\H$-separable for $r=1,\dots,n$  was shown in \cite[Lemma~5.13]{EMRT2018}.

Now consider the minimal $\D$-class $D_0$. Note that it is a $B(n)\times B(n)$ rectangular band, where~$B(
n)$ is the $n$th Bell number.
\JEnew{Fix an $\L$-class $L\subseteq D_0$ and an $\R$-class $R\subseteq D_0$.}
\NRrev{Neither $L$ nor $R$ is an $\H$-class} (for $n\geq2$), so we want to show that $\H_L$ and $\H_R$ are the unique maximal  congruences on
${}^{\P_n}\!L$ and $R^{\P_n}$ respectively.
We will show the latter statement, and the former can be deduced by duality.
So, let $\alpha,\beta,\gamma,\delta\in R$ be arbitrary, with $\alpha\neq\beta$.
We want to show that $(\gamma,\delta)\in(\alpha,\beta)^\sharp$.
We do this by showing that $(\gamma,\theta),(\delta,\theta)\in  (\alpha,\beta)^\sharp$, where $\theta\in R$ is the unique partition with singleton lower blocks, and using transitivity.
Only the first of these needs to be proved, as the second will be analogous.
If~$\gamma=\theta$ there is nothing to prove. Otherwise, $\theta$ has at most $n-1$ lower blocks.
In the proof of~\cite[Proposition 5.8]{EMRT2018} it was shown that, swapping~$\alpha$ and~$\beta$ if necessary, there exists $\zeta\in\P_n$ such that
$\gamma=\gamma\alpha\zeta$ and $\gamma':=\gamma\beta\zeta$ has one more lower block than~$\gamma$.
Since $\alpha,\beta,\gamma$ are all $\R$-related in the present consideration, and since all are idempotents, we have $\gamma\alpha=\alpha$ and $\gamma\beta=\beta$,
and so $(\gamma,\gamma')=(\alpha\cdot\zeta,\beta\cdot\zeta)\in(\alpha,\beta)^\sharp$.
Repeating this argument, and using transitivity, we arrive at $(\gamma,\theta)\in(\alpha,\beta)^\sharp$,  as required.

We can now deploy Corollaries \ref{co:S3}\ref{it:S31} and \ref{co:Sr3}\ref{it:Sr31}
to obtain
\[
\Ht(\LCong(\P_n))=\Ht(\RCong(\P_n))= B(n) + \displaystyle{\sum_{r=0}^n \sum_{k=r}^nS(n,k)\binom kr \Ht(\Sub(\S_r))}.
\]
For the two-sided congruences, we follow the same argument as for $\T_n$, with two small differences: first, we have an extra $\D$-class, namely $D_0$; and secondly, the extra additive contribution in~\eqref{eq:HtSJS} is $2$, instead of $1$. Therefore we obtain:
\[
\Ht(\Cong(\P_n))=3n+1 \qquad\text{for $n\geq4$.}
\]
Again, this could also be obtained using the description of the lattice $\Cong(\P_n)$ from \cite{EMRT2018}.

\subsection{Remarks concerning further semigroups}

The calculations for the remaining semigroups from our list are all entirely analogous to what we have just seen for $\T_n$ and $\P_n$, and we will not go through them in any kind of detail. Instead, we present the required input data, alongside  references to where it can be found in the literature, in Table \ref{tab:data}, and then we present the actual formulae for the heights of the lattices in Table~\ref{tab:formulae}.  To aid the reader in both inspecting the tables, and performing the calculations for additional semigroups not given in the tables, we give below some relevant commentary.
Towards the end of the section we will discuss two further monoids not included in the tables.

\vspace{2mm}\noindent
\textbf{Combinatorial notation.}
Various standard combinatorial sequences and arrays appear in the two tables. Here is a quick reference:
$S(n,r)$, Stirling numbers of the second kind;
$B(n)$, Bell numbers;
$C_n$, Catalan numbers;
$n!!$, double factorials;
$\bigl[\begin{smallmatrix} n\\ r\end{smallmatrix} \bigr]_q$, Gaussian binomial coefficients.

\vspace{2mm}\noindent
\textbf{Indexing of $\D$-classes.}
As we have already seen in the case of $\P_n$, the $\D$-classes in some of our monoids are naturally indexed by some sets of integers that are not necessarily of the form $\{1,\dots,n\}$, as in our theoretical results.
The starkest example of this is perhaps the Brauer monoid $\B_n$ and the Temperley--Lieb monoid $\TL_n$, whose $\D$-classes are indexed by all integers from~$[n]$ of the same parity as $n$. To aid the translation, we have included a column in Table \ref{tab:data} to explicitly state the indexing. We have also added a column to specify the minimal $\D$-class: this is always the $\D$-class with the smallest index, i.e.~$D_0$ or $D_1$.

\vspace{2mm}\noindent
\textbf{Sch\"{u}tzenberger group contributions.}
In all our examples, with the exception of the general linear monoid $\M(n,q)$, the Sch\"{u}tzenberger group $\Gamma_r$ of any $\H$-class in $D_r$ is \JErev{either $\S_r$ or trivial}.
For the symmetric group, its subgroup height is given in Theorem~\ref{thm:HtSn}, and its normal subgroup height in \eqref{eq:htnsr}.
\JErev{In some semigroups we will also encounter the symmetric group $\S_0$, which is trivial and has (normal) subgroup height $1$.}

The case where all Sch\"{u}tzenberger groups are trivial, as for example for $\O_n$ and $\TL_n$, leads to significant simplifications, as follows:

\begin{rem}
\label{re:aper}
Suppose $S$ is an $\H$-trivial semigroup, so that all its Sch\"{u}tzenberger groups are trivial.
Then the expression $\sum_{r=1}^n m_r\Ht(\Sub(\Gamma_r))$ in Corollary \ref{co:S3} simplifies to just $\sum_{r=1}^n m_r=|S/\L|$, the number of $\L$-classes of $S$.
Similarly, the corresponding sum in Corollary \ref{co:Sr3} becomes~$|S/\R|$, the number of $\R$-classes in $S$,
while the sum $\sum_{r=1}^n \Ht(\NSub(\Gamma_r))$ in Corollary~\ref{co:JS3} becomes just $n=|S/\D|$, the number of $\D$-classes.
\end{rem}

The Sch\"{u}tzenberger groups in the $\D$-class $D_r$ of the general linear monoid $\M(n,q)$ are isomorphic to the general linear group $\GL(r,q)$.
To the best of our knowledge, the heights of their subgroup lattices have not been determined, although \cite{Sol1,Sol2,Sol3}
discuss in detail the heights of closely related groups of Lie type. 
The normal subgroups of $\GL(r,q)$ are known: with a small number of exceptions, they are precisely the subgroups of the group of scalar matrices, and the subgroups containing the special linear group; see \cite[Theorem 4.9]{ArtinBook}. It is not clear whether this description is sufficient to give a closed formula for the height of the lattice of normal subgroups of  $\GL(r,q)$; see \cite[Figure 11.1]{ER2023}.
In our formulae, therefore, we keep the terms $\Ht(\Sub(\GL(r,q)))$ and $\Ht(\NSub(\GL(r,q)))$.

\vspace{2mm}\noindent
\textbf{Separability.}
In all our examples, all non-minimal $\D$-classes $D_r$ are $\H$-separable. This can be easily verified by the reader by considering the matrices $M(D_r)$ and checking they are row- and column-faithful.
Alternatively, this has been proved in \cite{ER2023} in the more general context of certain related categories.

There is one general, theoretical situation where $\H$-separability is guaranteed from the outset, namely for inverse semigroups with finitely many idempotents, where each matrix $M(D)$ is always equivalent to the identity matrix up to permutation of rows and columns.
The reason for this is that in an inverse semigroup every $\L$- and every $\R$-class contains precisely one idempotent.
In particular, we have $|D/\L|=|D/\R|=|E(D)|$, the number of idempotents in $D$.
The above is true for \emph{all} $\D$-classes, including the minimal one, which is always necessarily a group.

\begin{cor}\label{co:in}
Let $S$ be an inverse semigroup with finitely many  $\D$-classes $D_1,\dots, D_n$,
and let $\Gamma_1,\dots,\Gamma_n$ be the Sch\"{u}tzenberger groups 
\JErev{(i.e.~maximal subgroups)} associated with their $\H$-classes, respectively. 
\begin{thmenumerate}
\item
If each $D_r$ has only finitely many idempotents then
\[
\Ht(\LCong(S))=\Ht(\RCong(S))=\sum_{r=1}^n |E(D_r)|\Ht(\Sub(\Gamma_r)).
\]
Moreover, this is equal to $|E(S)|$ if $S$ is $\H$-trivial.
\item
If $S$ is stable then
\[
\Ht(\Cong(S))=\sum_{r=1}^n \Ht(\NSub(\Gamma_r)).
\]
Moreover, this is equal to $n=|S/\D|$ if $S$ is $\H$-trivial.  \epfres
\end{thmenumerate}
\end{cor}

As a consequence, if $S$ is a finite semilattice then \[\Ht(\LCong(S))=\Ht(\RCong(S))=\Ht(\Cong(S))=|S|.\]

\vspace{2mm}\noindent
\textbf{Relationship to classification results.}
Just as was the case with $\T_n$ and $\P_n$ above, all the two-sided congruences have been described for all the semigroups listed in Table \ref{tab:formulae}. The oldest and most well-known such results are due to Malcev for the full transformation monoid \cite{Malcev1952} and the general linear monoid \cite{Malcev1953}, and Liber for the symmetric inverse monoid \cite{Liber1953}. These results have been gathered under a common methodological framework in \cite{ER2023},
where the reader can also find references to the historical individual proofs.
From these descriptions, it is possible to infer the formulae for the height of the congruence lattice by straightforward inspection, with the possible exception of the general linear monoid, where perhaps the structure of the congruence lattice is not entirely obvious from the description, but see \cite[Section 11.4]{ER2023} for a discussion of this point.

The reason we have opted to present these values in the way we did is two-fold:
1)  The height is computed without the prior of knowledge of what all the congruences are, 
so the present method can be regarded as being conceptually simpler than going via classification theorems. 
2) The computation follows the same lines as that for one-sided congruences, emphasising the methodological proximity of the two.

In contrast with two-sided congruences, very little is known, by way of classification, for one-sided congruences of the semigroups discussed here. One exception is the symmetric inverse monoid $\I_n$, where useful detailed information about its left and right congruences can be obtained by utilising Brookes' inverse kernel approach; see \cite{Brookes2023} for the general theory and \cite{Brookesthesis} for applications to $\I_n$. Combining this with \cite[Theorem 7.4]{CST1989} may offer an alternative route towards calculating $\Ht(\LCong(\I_n))=\Ht(\RCong(\I_n))$.

\vspace{2mm}\noindent
\textbf{Miscellaneous notes on the table entries.}
The formulae for the heights of the lattices of one-sided congruences for $\O_n$ and $\TL_n$ given in Table \ref{tab:formulae}
are not identical to the sums obtained using the information from Table \ref{tab:data}.
The reason is that the latter sums simplify, and proving the equality of the two expressions is an elementary exercise.
Alternatively, we can recall that~$\O_n$ and~$\TL_n$ are $\H$-trivial, and use  Remark \ref{re:aper}.
To do this we need the total numbers of $\L$- and $\R$-classes in the monoids: for $\O_n$ these are $2^{n}-1$ and $2^{n-1}$ respectively, and for $\TL_n$ they are $\binom{n}{\lfloor n/2\rfloor}$ and $\binom{n}{\lfloor n/2\rfloor}$.

The formulae for the height of the congruence lattices of semigroups whose maximal subgroups are the symmetric groups are valid for $n\geq 4$. For smaller values of $n$ the heights deviate slightly from those formulae, due to the initial irregularity of 
the sequence $(\Ht(\NSub(\S_r)))_{r\geq 0}$; see \eqref{eq:htnsr}.
These values can be found in Table \ref{tab:small}, which gives computed values for $n\leq10$. The Semigroups package for GAP \cite{GAP4,Semigroups} is capable of computing all the congruences for some small values of $n$, which provides a verification of the given numbers.
As the family $\M(n,q)$ of linear monoids depends on two parameters, we have not included any numerical data for it in Table \ref{tab:small}.

In the formulae for the heights of the one-sided congruence lattices of $\B_n$ and $\TL_n$ we have made a couple of indexing simplifications in the transition between Table \ref{tab:data} and Table \ref{tab:formulae}.
We have changed the summation parameter from $r$ to $k$, where $r=n-2k$.
We have also used the fact that $2\lfloor \frac{n-1}{2}\rfloor +1$ equals  $n-1$ when $n$ is even and $n$ when $n$ is odd.

\vspace{2mm}\noindent
\textbf{Partial Brauer monoid.}
Another diagram monoid, not listed in Tables \ref{tab:data} and \ref{tab:formulae}, is the \emph{partial Brauer monoid} $\PB_n$; it consists of all partitions in $\P_n$ with block-sizes at most $2$.  One of the consequences of \cite[Theorems 5.4 and 6.1]{EMRT2018} is that the lattices~$\Cong(\P_n)$ and $\Cong(\PB_n)$ are isomorphic, and hence have the same height.  The analogue is not true for one-sided congruence lattices (apart from trivial exceptions for small $n$).  Indeed, $\PB_n$ again has $\D$-classes $D_0,D_1,\ldots,D_n$, given by ranks, and Corollary~\ref{co:S3} applies; see \cite[Section~6]{EMRT2018} for the relevant separation properties.  This gives a formula for $\Ht(\LCong(\PB_n))$ in terms of the parameters 
$|D_r/\L|=\binom{n}{r}a(n-r)$, where $a(k)$ stands for the number of involutions on $[k]$; see 
\cite[Proposition 4.6]{DDE2021}. 
Since these are strictly smaller than the corresponding parameters for~$\P_n$ itself, we have $\Ht(\LCong(\PB_n))<\Ht(\LCong(\P_n))$ for~$n\geq2$.  
\JEnew{The same is true for the right congruence lattices.}

\vspace{2mm}\noindent
\textbf{Monoids of uniform block bijections.}
For all the monoids discussed so far, the posets of $\J(=\D)$-classes happen to be chains.
This is not a requirement for our formulas to apply. We conclude by briefly discussing a monoid with a more complicated $\J$-structure, namely
the monoid $\F_n^*$ of all \emph{uniform} block bijections from $\I_n^*$, i.e.~those for which each block contains the same number of upper and lower points. The $\D$-classes of $\F_n^*$ are of the form $D_\mu$, where $\mu\vdash n$ is an \emph{integer partition} of $n$, i.e.~a tuple ${\mu=(m_1,\ldots,m_k)}$ for some $k\geq1$, where $m_1\geq\cdots\geq m_k\geq1$ are integers, and $m_1+\cdots+m_k=n$. Given such a $\mu\vdash n$, the $\D$-class $D_\mu$ contains all block bijections of rank~$k$, whose blocks contain $m_1,\ldots,m_k$ upper (and lower) points. As explained in \cite[Section~3]{FL1998}, the parameters associated to such a $\D$-class are as follows. First, for $1\leq i\leq n$, we write $\mu_i$ for the number of entries of~$\mu$ equal to $i$ (so $n = \sum_{i=1}^ni\mu_i$). The number of $\L$- and $\R$-classes in~$D_\mu$ is equal to the number of set partitions of $[n]$ whose blocks have sizes $m_1,\ldots,m_k$, which is well known to be equal to 
$n!/\prod_{i=1}^n\mu_i!(i!)^{\mu_i}$.
The Sch\"utzenberger group of any $\H$-class in $D_\mu$ is isomorphic to the \emph{Young subgroup} $\S_\mu:=\S_{\mu_1}\times\cdots\times\S_{\mu_n}$ of $\S_k$. Applying Corollary \ref{co:in} yields
\begin{align*}
\Ht(\LCong(\F_n^*)) = \Ht(\RCong(\F_n^*)) &= \sum_{\mu\vdash n} \frac{n!\Ht(\Sub(\S_\mu))}{\prod_{i=1}^n\mu_i!(i!)^{\mu_i}},\\
 \Ht(\Cong(\F_n^*))  &=\sum_{\mu\vdash n}\Ht(\NSub(\S_\mu))  .
\end{align*}
To the best of the authors' knowledge 
neither the one-sided nor two-sided congruences of $\F_n^*$ have been classified.
\bigskip

\noindent
\NRrev{\textbf{Acknowledgements.}
We are grateful to three anonymous referees for their insightful comments and helpful suggestions.}

\footnotesize
\def\bibspacing{-1.1pt}
\bibliography{biblio}
\bibliographystyle{abbrv}

\begin{sidewaystable}
\[
\begin{array}{|c||c|c|c|c||c|c|c||c|}
\hline
\multirow{2}{20mm}{\textbf{Semigroup}} &\multicolumn{4}{c||}{ \textbf{All \boldmath{$\D$}-classes \boldmath{$D_r$}}} & \multicolumn{3}{c||}{\textbf{Minimal \boldmath{$\D$}-class \boldmath{$D_r$}} } & \multirow{2}{20mm}{\textbf{Reference}}
\\ \cline{2-8}
& \text{\boldmath{$r$}} \textbf{ range} & \text{\boldmath{$|D_r/\L|$}} & \text{\boldmath{$|D_r/\R|$}} &
\text{\boldmath{$\Gamma_r$}} & 
\text{\boldmath{$r$}} & \text{\boldmath{$|D_r/\L|$}} & \text{\boldmath{$|D_r/\R|$}}&
\\ \hline\hline
\rule[-4.5mm]{0mm}{11mm}\T_n & 1\leq r\leq n & \displaystyle{\binom{n}{r}} &  S(n,r) & \S_r & 
 r=1 & n & 1 & \text{\cite[Section 4.6]{GMbook}} \\ \hline
\rule[-4.5mm]{0mm}{11mm} \PT_n &0\leq r\leq n&\displaystyle{\binom{n}{r}}  & S(n+1,r+1) &\S_r&r=0&1&1&\text{\cite[Section 4.6]{GMbook}}\\ \hline
\rule[-4.5mm]{0mm}{11mm} \I_n &0\leq r\leq n&\displaystyle{\binom{n}{r}} &\displaystyle{\binom{n}{r}} &\S_r&r=0&1&1&\text{\cite[Section 4.6]{GMbook}} \\ \hline
 \rule[-4.5mm]{0mm}{11mm} \O_n &1\leq r\leq n&\displaystyle{\binom{n}{r}} &\displaystyle{\binom{n-1}{r-1}} &1&r=1&n&1&\text{\cite[Section 2]{Garb1994}}\\ \hline
\rule[-5.5mm]{0mm}{1mm} \M(n,q) &0\leq r\leq n&{}_{\phantom{q}}\biggl[\begin{matrix} n\\ r\end{matrix}\biggr]_q&
{}_{\phantom{q}}\biggl[\begin{matrix} n\\ r\end{matrix} \biggr]_q&\GL(r,q)&r=0&1&1&\text{\cite[Section 3]{DE2018}}\\ \hline
\rule[-5.5mm]{0mm}{11mm}\P_n & 0\leq r\leq n & \displaystyle{\sum_{k=r}^nS(n,k)\binom kr} &\displaystyle{\sum_{k=r}^nS(n,k)\binom kr}&\S_r&r=0&B(n)&B(n)&  \text{\cite[Section 7]{EG2017}}  \\ \hline
\B_n & \begin{array}{c} 0\leq r\leq n \\ r\equiv n\!\!\pmod{2}\end{array} &\displaystyle{\binom{n}{r}(n-r-1)!!}&\displaystyle{\binom{n}{r}(n-r-1)!!}&\S_r&\begin{array}{l} r=0\\ r=1\end{array}&
\begin{array}{c} (n-1)!! \\ n!! \end{array} &\begin{array}{c} (n-1)!! \\ n!! \end{array}& \text{\cite[Section 8]{EG2017}} \\ \hline
\TL_n &\begin{array}{c} 0\leq r\leq n \\ r\equiv n\!\!\pmod{2}\end{array}&\displaystyle{\frac{r+1}{n+1}\binom{n+1}{\frac{n-r}{2}}}&\displaystyle{\frac{r+1}{n+1}\binom{n+1}{\frac{n-r}{2}}}&1&r=0\text{ or } 1&C_{\lceil\frac{n}{2}\rceil}&C_{\lceil\frac{n}{2}\rceil}&\text{\cite[Section 9]{EG2017}} \\ \hline
\I_n^* &1\leq r\leq n&S(n,r)&S(n,r) &\S_r &r=1&1&1&\text{\cite[Section 2]{FL1998}} \\ \hline
\end{array}
\]
\caption{The data concerning Green's structure for some semigroups of transformations and partitions needed to compute the heights of their lattices of left, right and two-sided congruences.}
\label{tab:data}
\end{sidewaystable}

\begin{sidewaystable}
\[
\begin{array}{|c||c|c|c|}
\hline
\textbf{\boldmath{$S$}} & \text{\boldmath{$\Ht(\LCong(S))$}} &\text{\boldmath{$\Ht(\RCong(S))$}} &
\text{\boldmath{$\Ht(\Cong(S))$}} 
\\ \hline\hline
\rule[-5.5mm]{0mm}{12.5mm}\T_n & \displaystyle{\sum_{r=1}^n \binom nr \Ht(\Sub(\S_r))}  & \displaystyle{1+\sum_{r=1}^n S(n,r) \Ht(\Sub(\S_r)) }& 3n-1\\ \hline
\rule[-5.5mm]{0mm}{12.5mm}\PT_n &\displaystyle{\sum_{r=0}^n \binom nr \Ht(\Sub(\S_r))} &\displaystyle{\sum_{r=0}^n S(n+1,r+1) \Ht(\Sub(\S_r)) }&3n-1\\ \hline 
\rule[-5.5mm]{0mm}{12.5mm}\I_n &\displaystyle{\sum_{r=0}^n \binom nr \Ht(\Sub(\S_r))}&\displaystyle{\sum_{r=0}^n \binom nr \Ht(\Sub(\S_r))}&3n-1\\ \hline
\O_n &2^n-1&2^{n-1}+1&n+1 \\ \hline
\rule[-5.5mm]{0mm}{12.5mm}\M(n,q) &\displaystyle{\sum_{r=0}^n {} \biggl[\begin{matrix} n\\ r\end{matrix}\biggr]_q \Ht(\Sub(\GL(r,q)))}&\displaystyle{\sum_{r=0}^n {} \biggl[\begin{matrix} n\\ r\end{matrix}\biggr]_q \Ht(\Sub(\GL(r,q)))}&
\displaystyle{\sum_{r=0}^n {} \Ht(\NSub(\GL(r,q)))}\\ \hline
\rule[-5.5mm]{0mm}{12.5mm}\P_n & B(n) + \displaystyle{\sum_{r=0}^n \sum_{k=r}^nS(n,k)\binom kr \Ht(\Sub(\S_r))}  & B(n) + \displaystyle{\sum_{r=0}^n \sum_{k=r}^nS(n,k)\binom kr \Ht(\Sub(\S_r))}  &3n+1  
\\ \hline
\rule[-5.5mm]{0mm}{13.5mm}\B_n &(2\big\lfloor\frac{n-1}{2}\big\rfloor +1)!! + \displaystyle{\sum_{k=0}^{\lfloor n/2 \rfloor}} \binom{n}{2k}(2k-1)!!\Ht(\Sub(\S_{n-2k})) &
(2\big\lfloor\frac{n-1}{2}\big\rfloor +1)!! +\displaystyle{\sum_{k=0}^{\lfloor n/2 \rfloor}} \binom{n}{2k}(2k-1)!!\Ht(\Sub(\S_{n-2k}))&
\displaystyle{3\Big\lfloor \frac{n}{2}\Big\rfloor+3}
\\ \hline
\rule[-5mm]{0mm}{11.5mm}\TL_n & C_{\lceil \frac{n}{2}\rceil} + \displaystyle{\binom{n}{\lfloor \frac{n}{2}\rfloor }} &C_{\lceil \frac{n}{2}\rceil} + \displaystyle{\binom{n}{\lfloor \frac{n}{2}\rfloor }} &  \displaystyle{\Big\lfloor \frac{n}{2}\Big\rfloor}+3\\ \hline
\rule[-5.5mm]{0mm}{12.5mm}\I_n^* &\displaystyle{\sum_{r=1}^n S(n,r) \Ht(\Sub(\S_r))}&\displaystyle{\sum_{r=1}^n S(n,r) \Ht(\Sub(\S_r))}&3n-2\\ \hline
\end{array}
\]

\caption{The formulae for the heights of the left, right and two-sided congruence lattices for some semigroups of transformations and partitions. Throughout we take $n\geq 4$, but some formulae remain valid for smaller $n$.
The actual values for small $n$ are given in Table \ref{tab:small}.}
\label{tab:formulae}

\end{sidewaystable}

\begin{table}
\[
\begin{array}{|c|c||c|c|c|c|c|c|c|c|c|c|c|}\hline
\textbf{Semigroup} & \textbf{Lattice} & \textbf{0}& \textbf{1}& \textbf{2}& \textbf{3}& \textbf{4}& \textbf{5}& \textbf{6}& \textbf{7}& \textbf{8}& \textbf{9}& \textbf{10}
\\ \hline\hline
\multirow{3}{*}{$\T_n$} & \LCong & 1 & 1 & 4 & 12 & 33 & 86 & 214 & 512 & 1189 & 2706 & 6080
\\ \cline{2-13}
&\RCong &1 & 1 & 4 & 11 & 39 & 163 & 756 & 3776 & 20056 & 112769 & 670321
\\ \cline{2-13}
&\Cong & 1 & 1 & 4 & 7 & 11 & 14 & 17 & 20 & 23 & 26 & 29
\\ \hline
\multirow{3}{*}{$\PT_n$} & \LCong & 1 & 2 & 5 & 13 & 34 & 87 & 215 & 513 & 1190 & 2707 & 6081
\\ \cline{2-13}
&\RCong &1 & 2 & 6 & 23 & 101 & 488 & 2549 & 14213 & 83780 & 518741 & 3362175
\\ \cline{2-13}
&\Cong & 1 & 2 & 4 & 7 & 11 & 14 & 17 & 20 & 23 & 26 & 29
\\ \hline
\multirow{3}{*}{$\I_n$} & \LCong & 1 & 2 & 5 & 13 & 34 & 87 & 215 & 513 & 1190 & 2707 & 6081
\\ \cline{2-13}
&\RCong &1 & 2 & 5 & 13 & 34 & 87 & 215 & 513 & 1190 & 2707 & 6081
\\ \cline{2-13}
&\Cong & 1 & 2 & 4 & 7 & 11 & 14 & 17 & 20 & 23 & 26 & 29
\\ \hline
\multirow{3}{*}{$\O_n$} & \LCong &1 & 1 & 3 & 7 & 15 & 31 & 63 & 127 & 255 & 511 & 1023 
\\ \cline{2-13}
&\RCong &1 & 1 & 3 & 5 & 9 & 17 & 33 & 65 & 129 & 257 & 513 
\\ \cline{2-13}
&\Cong &1 & 1 & 3 & 4 & 5 & 6 & 7 & 8 & 9 & 10 & 11
\\ \hline
\multirow{3}{*}{$\P_n$} & \LCong & 1 & 2 & 9 & 35 & 164 & 881 & 5260 & 34215 & 239263 & 1782109 & 14043708
\\ \cline{2-13}
&\RCong &1 & 2 & 9 & 35 & 164 & 881 & 5260 & 34215 & 239263 & 1782109 & 14043708
\\ \cline{2-13}
&\Cong & 1 & 2 & 6 & 9 & 13 & 16 & 19 & 22 & 25 & 28 & 31
\\ \hline
\multirow{3}{*}{$\B_n$} & \LCong & 1 & 1 & 3 & 9 & 23 & 66 & 202 & 659 & 2307 & 8238 & 32008
\\ \cline{2-13}
&\RCong &1 & 1 & 3 & 9 & 23 & 66 & 202 & 659 & 2307 & 8238 & 32008
\\ \cline{2-13}
&\Cong & 1 & 1 & 3 & 6 & 9 & 9 & 12 & 12 & 15 & 15 & 18
\\ \hline
\multirow{3}{*}{$\TL_n$} & \LCong & 1&1& 2& 5&8& 15& 25& 49& 84& 168& 294
\\ \cline{2-13}
&\RCong &1&1& 2& 5&8& 15& 25& 49& 84& 168& 294
\\ \cline{2-13}
&\Cong & 1&1& 2& 4& 5& 5& 6& 6& 7& 7& 8
\\ \hline
\multirow{3}{*}{$\I_n^*$} & \LCong & 1 & 1 & 3 & 10 & 38 & 162 & 755 & 3775 & 20055 & 112768 & 670320
\\ \cline{2-13}
&\RCong &1 & 1 & 3 & 10 & 38 & 162 & 755 & 3775 & 20055 & 112768 & 670320
\\ \cline{2-13}
&\Cong & 1 & 1 & 3 & 6 & 10 & 13 & 16 & 19 & 22 & 25 & 28
\\ \hline
\end{array}
\]

\caption{Heights of lattices of left, right and two-sided congruence lattices for the monoids of transformations and partitions from Table \ref{tab:formulae}. }

\label{tab:small}

\end{table}

\

\end{document}